\documentclass[aop,preprint]{imsart}
\usepackage[width=6.4in,height=10in,center]{crop}
\makeatletter
   \def\journal@name{}
   \def\journal@url{}
\makeatother

\RequirePackage[OT1]{fontenc}
\usepackage{mathtools}
\RequirePackage{amsthm,mathtools}
\usepackage{microtype}
\usepackage[pdftitle={The Conformal Loop Ensemble Nesting Field},
            pdfauthor={Jason Miller, Samuel S. Watson, and David B. Wilson},
            bookmarks=true,bookmarksopen=true,bookmarksopenlevel=3,unicode,colorlinks=false]{hyperref}

\numberwithin{equation}{section}
\numberwithin{figure}{section}

\usepackage{etoolbox}
\usepackage{comment}
\usepackage{appendix}
\usepackage{graphicx}
\usepackage{subfigure}
\usepackage{enumerate}
\usepackage{comment}
\usepackage{colonequals}
\usepackage{color}
\usepackage[all]{hypcap}
\usepackage{soul}
\usepackage{bm}
\usepackage{mathpazo_modified}\usepackage{amsfonts}
\usepackage{amssymb}
\usepackage{yhmath}
\usepackage{accents}

\startlocaldefs

\newtheorem{theorem}{Theorem}[section]
\newtheorem{lemma}[theorem]{Lemma}
\newtheorem{proposition}[theorem]{Proposition}
\newtheorem{corollary}[theorem]{Corollary}

\newcommand{\area}{\operatorname{area}}
\newcommand{\const}{\textrm{const}}
\newcommand{\eps}{\varepsilon}
\newcommand{\rc}[1]{\accentset{\circ}{#1}}
\renewcommand{\P}{\mathbb{P}}

\renewcommand{\C}{\mathbb{C}}
\newcommand{\D}{\mathbb{D}}
\newcommand{\E}{\mathbb{E}}

\newcommand{\N}{\mathbb{N}}
\newcommand{\Z}{\mathbb{Z}}

\newcommand{\R}{\mathbb{R}}
\newcommand{\bound}{a}
\newcommand{\T}{\mathbb{T}}

\newcommand{\Sc}{\mathcal{S}}

\newcommand{\Hloc}{H_{\operatorname{loc}}}

\newcommand{\lp}{\nu} \newcommand{\lptyp}{\lp_{\mathrm{typical}}} \newcommand{\pp}{\lambda}                        \newcommand{\mgfparam}{\lambda}  \newcommand{\gffparam}{\sigma}

\newcommand{\CL}{\mathcal{L}}

\DeclareMathOperator{\dist}{dist}

\DeclareMathOperator{\interior}{int}

\DeclareMathOperator{\cov}{Cov}
\DeclareMathOperator{\Cov}{Cov}
\DeclareMathOperator{\Var}{Var}

\DeclareMathOperator{\supp}{supp}

\DeclareMathOperator{\SLE}{SLE}
\DeclareMathOperator{\CLE}{CLE}
\DeclareMathOperator{\confrad}{CR}

\DeclareMathOperator{\inrad}{inrad}

\newcommand{\Loop}{\mathcal{L}}
\newcommand{\Loopcount}{\mathcal{N}}
\newcommand{\SLoopcount}{\mathcal{S}}

\newcommand{\TLoopsum}{\widetilde{\mathcal{S}}}

\newcommand{\given}{\,|\,}

\newcommand{\one}{{\bf 1}}

\newcommand{\ol}{\overline}

\newcommand{\wh}{\widehat}
\newcommand{\wt}{\widetilde}
\newcommand{\ang}[1]{\langle #1 \rangle}

\newcommand{\arXiv}[1]{\href{http://arxiv.org/abs/#1}{arXiv:#1}}
\renewcommand{\arxiv}[1]{\href{http://arxiv.org/abs/#1}{#1}}

\definecolor{dgreen}{rgb}{0.0,0.5,0.0}

\definecolor{dorange}{rgb}{1.0,0.5,0.0}

\newtoggle{BW}
\togglefalse{BW}

\iftoggle{BW}{

}{

}
\numberwithin{equation}{section}

\usepackage[pagewise,mathlines]{lineno}

\newcommand*\patchAmsMathEnvironmentForLineno[1]{  \expandafter\let\csname old#1\expandafter\endcsname\csname #1\endcsname
  \expandafter\let\csname oldend#1\expandafter\endcsname\csname end#1\endcsname
  \renewenvironment{#1}     {\linenomath\csname old#1\endcsname}     {\csname oldend#1\endcsname\endlinenomath}}\newcommand*\patchBothAmsMathEnvironmentsForLineno[1]{  \patchAmsMathEnvironmentForLineno{#1}  \patchAmsMathEnvironmentForLineno{#1*}}\AtBeginDocument{\patchBothAmsMathEnvironmentsForLineno{equation}\patchBothAmsMathEnvironmentsForLineno{align}\patchBothAmsMathEnvironmentsForLineno{flalign}\patchBothAmsMathEnvironmentsForLineno{alignat}\patchBothAmsMathEnvironmentsForLineno{gather}\patchBothAmsMathEnvironmentsForLineno{multline}}

\usepackage{xr}
\usepackage{textgreek}

\begin{document}

\begin{frontmatter}
\vspace*{-64pt}
\title{The~conformal~loop~ensemble nesting field}
  \runtitle{The CLE nesting field}
\begin{aug}
    \runauthor{Jason Miller, Samuel S.~\!Watson, and David B.~\!Wilson}
    \author{Jason Miller, Samuel S.~\!Watson, and David B.~\!Wilson}
\affiliation{Microsoft Research and Massachusetts Institute of Technology}
\end{aug}
\date{}

\makeatletter{}\begin{abstract}
  The conformal loop ensemble $\CLE_\kappa$ with parameter $8/3 < \kappa <
  8$ is the canonical conformally invariant measure on countably infinite
  collections of non-crossing loops in a simply connected domain.  We show
  that the number of loops surrounding an $\eps$-ball (a random function of
  $z$ and $\eps$) minus its expectation converges almost surely as $\eps\to
  0$ to a random conformally invariant limit in the space of distributions,
  which we call the nesting field. We generalize this result by assigning
  i.i.d.\ weights to the loops, and we treat an alternate notion of
  convergence to the nesting field in the case where the weight
  distribution has mean zero. We also establish estimates for moments of
  the number of CLE loops surrounding two given points.
\end{abstract}

\setattribute{keyword}{AMS}{AMS 2010 subject classifications:}

\end{frontmatter}

\maketitle

\thispagestyle{empty}
{\renewcommand{\contentsname}{}
\vspace*{-42pt}
\small \tableofcontents}

\section{Introduction}
\label{sec::introduction}

\enlargethispage{36pt}
The conformal loop ensemble $\CLE_\kappa$ for $\kappa \in (8/3,8)$ is the
canonical conformally invariant measure on countably infinite collections
of non-crossing loops in a simply connected domain \label{not::domD} $D
\subsetneq \C$ \cite{SHE_CLE,SW_CLE}.  It is the loop analogue of
$\SLE_\kappa$, the canonical conformally invariant measure on non-crossing
paths.  Just as $\SLE_\kappa$ arises as the scaling limit of a single
interface in many two-dimensional discrete models, $\CLE_\kappa$ is a
limiting law for the joint distribution of all of the interfaces.
Figures~\ref{fig::On} and~\ref{fig::FK} show two discrete loop models
believed or known to have $\CLE_\kappa$ as a scaling limit.
Figure~\ref{fig::CLE_examples} illustrates these scaling limits for several
values of $\kappa$.

\begin{figure}[b!]
\begin{center}
\subfigure[Site percolation.]{\includegraphics[width=.32\textwidth]{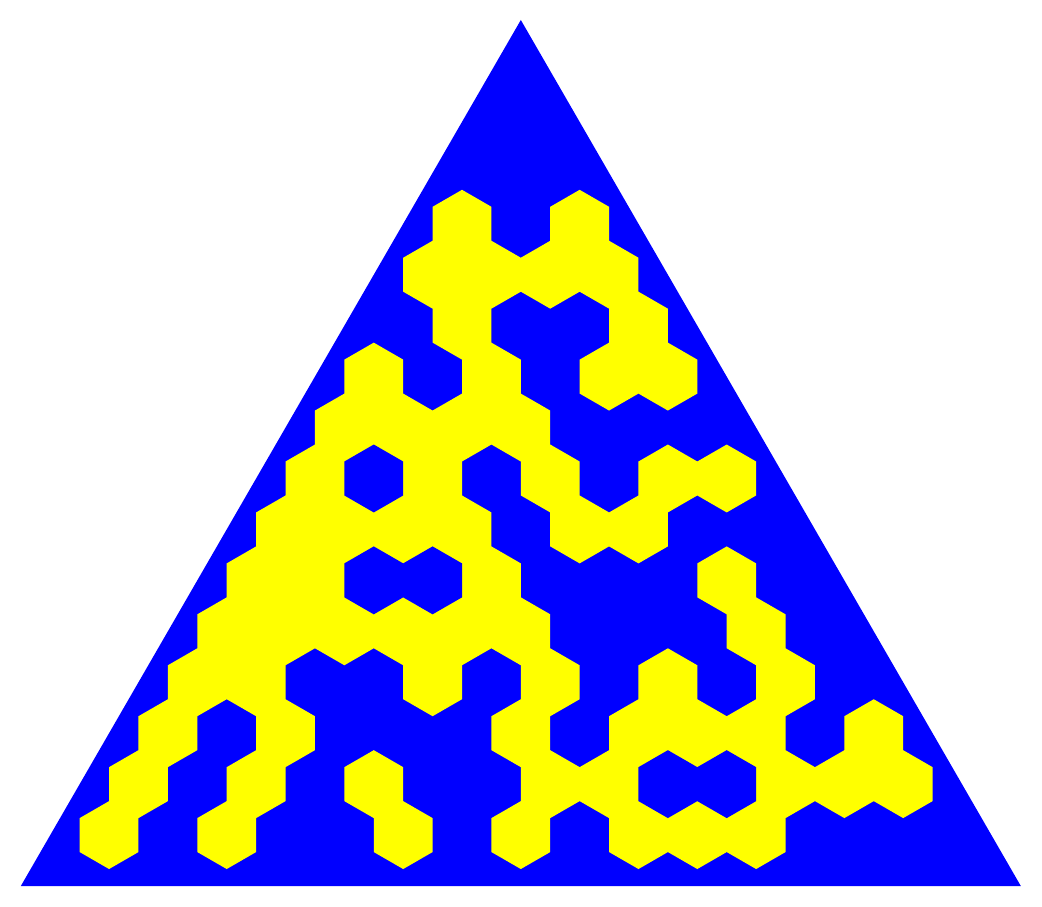}}
\hfill
\subfigure[$O(n)$ loop model.  Percolation corresponds to $n=1$ and $x=1$, which
is in the dense phase.]{\includegraphics[width=.32\textwidth]{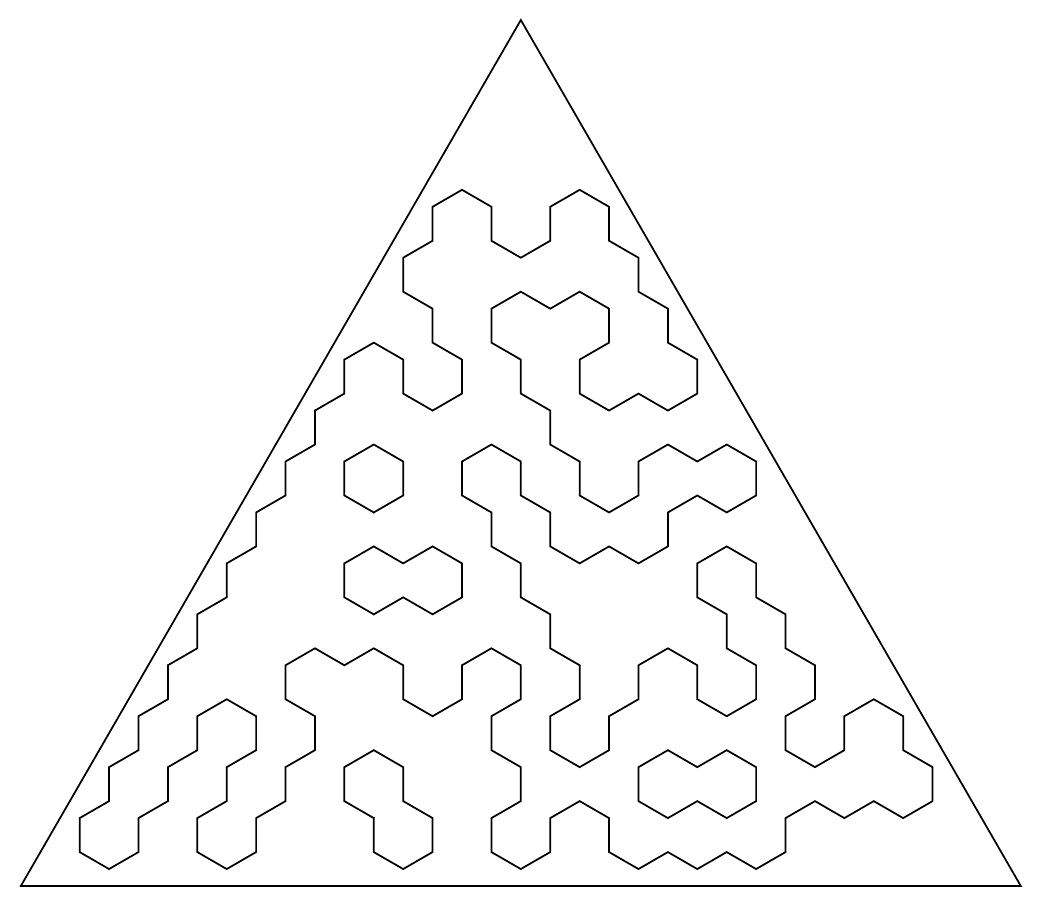}}
\hfill
\subfigure[Area shaded by nesting of loops.]{\includegraphics[width=.32\textwidth]{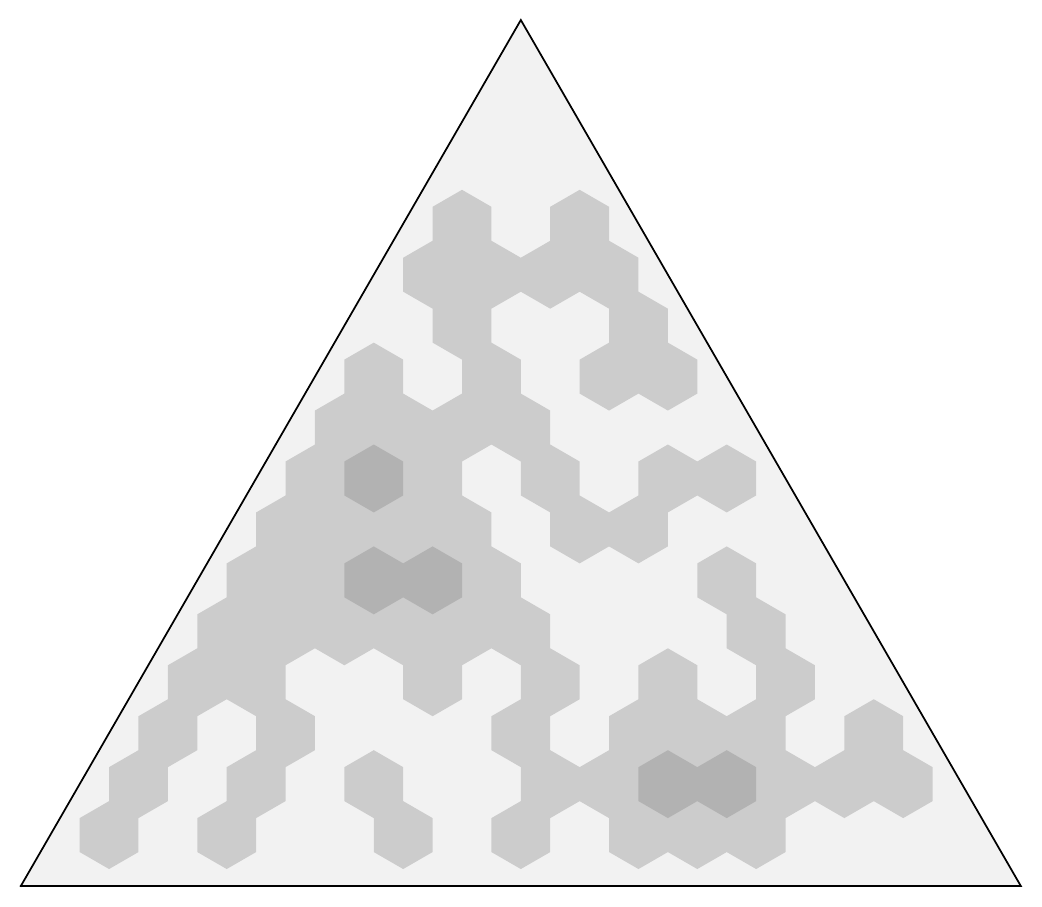}}
\end{center}
\caption{\label{fig::On} Nesting of loops in the $O(n)$ loop model.
  Each $O(n)$ loop configuration has probability proportional to $x^{\text{total length of loops}} \times n^{\text{\# loops}}$.
  For a certain critical value of $x$,
  the $O(n)$ model for $0\leq n\leq 2$ has a ``dilute
  phase'', which is believed to
  converge $\CLE_\kappa$ for $8/3<\kappa\leq 4$ with $n=-2\cos(4\pi/\kappa)$.
  For $x$ above this critical value, the $O(n)$ loop model
  is in a ``dense phase'',
  which is believed to converge to $\CLE_\kappa$ for $4\leq\kappa\leq 8$,
  again with $n=-2\cos(4\pi/\kappa)$.  See \cite{KN04} for further background.}
\end{figure}

\begin{figure}
\begin{center}
\subfigure[Critical FK bond configuration.  Here $q=2$.]{\includegraphics[width=.32\textwidth]{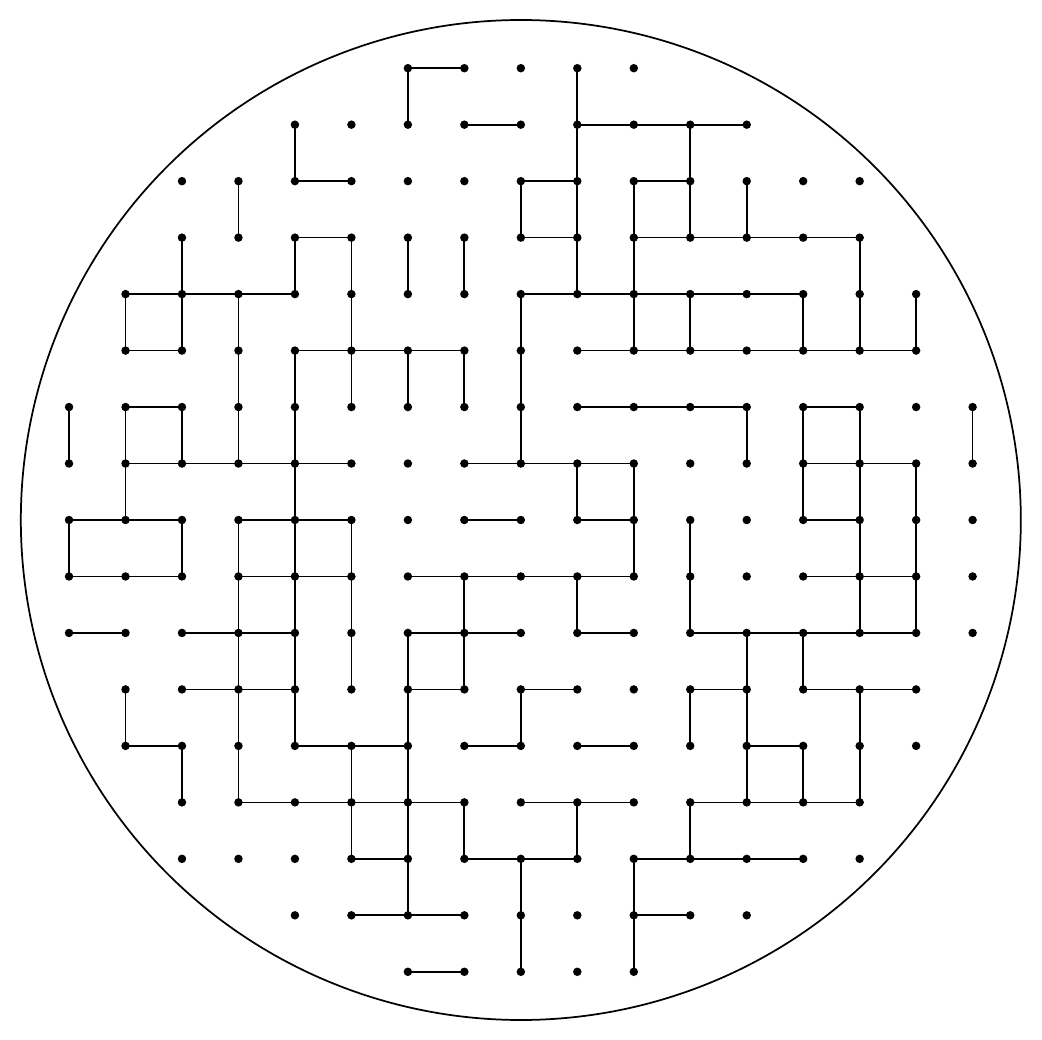}}
\hfill
\subfigure[Loops separating FK clusters from dual clusters.]{\includegraphics[width=.32\textwidth]{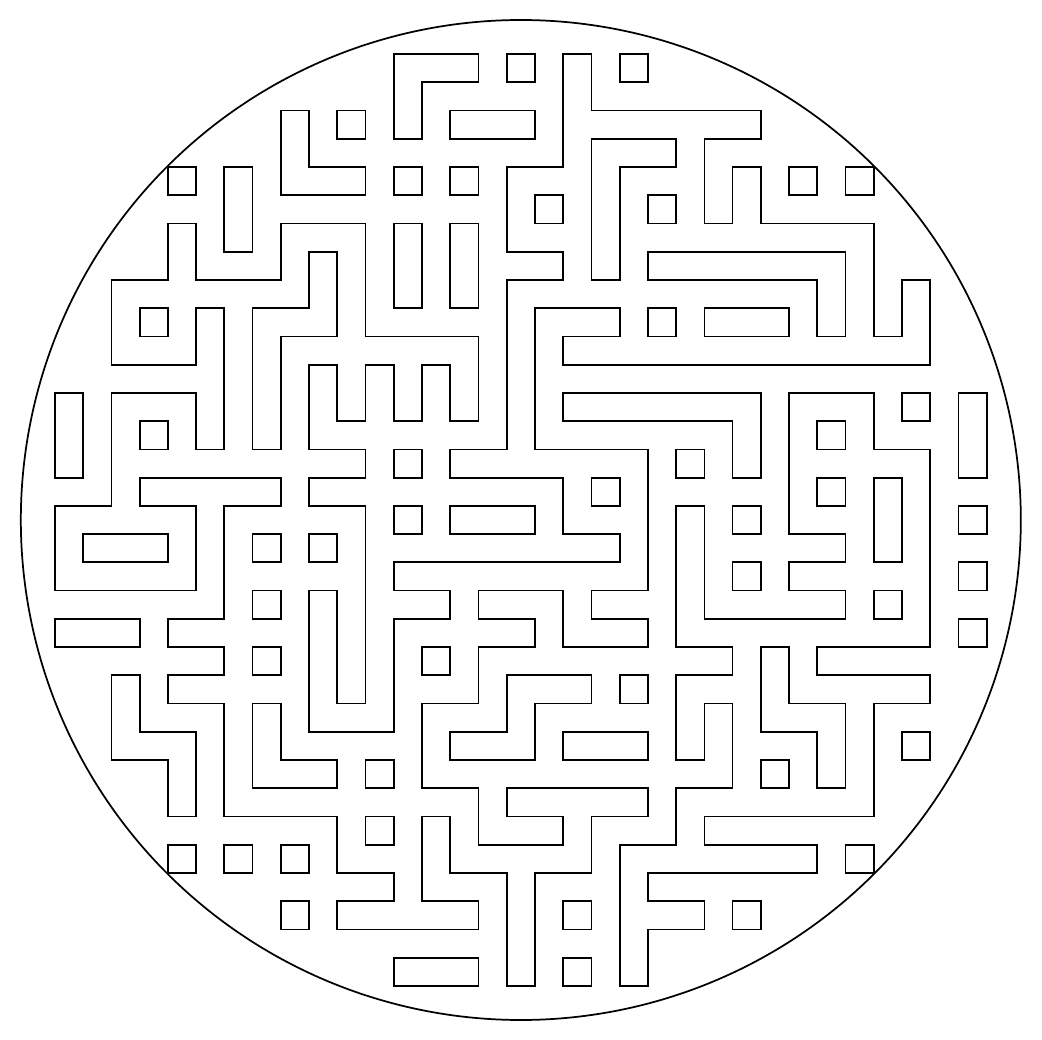}}
\hfill
\subfigure[Area shaded by nesting of loops.]{\includegraphics[width=.32\textwidth]{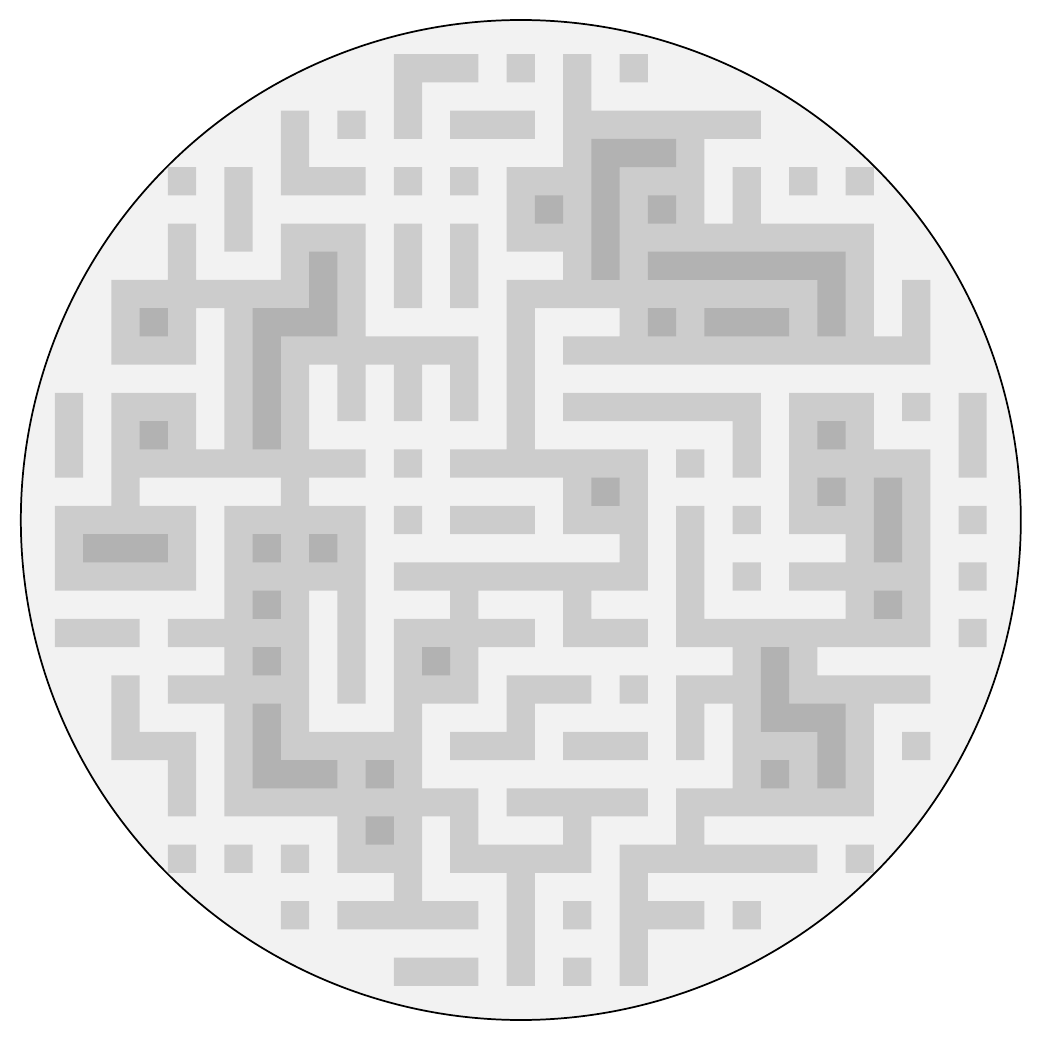}}
\end{center}
\caption{\label{fig::FK} Nesting of loops separating critical
  Fortuin-Kasteleyn (FK) clusters from dual clusters.
  Each FK bond configuration has probability proportional to
  $(p/(1-p))^{\text{\# edges}} \times q^{\text{\# clusters}}$ \cite{FK}, where
  there is believed to be a critical point at $p=1/(1+1/\sqrt{q})$ (proved for $q\geq 1$ \cite{BDC}).
  For $0\leq q\leq 4$, these loops are believed to have the
  same large-scale behavior as the $O(n)$ model
  loops for $n=\sqrt{q}$ in the dense phase, that is, to
  converge to $\CLE_\kappa$ for $4\leq\kappa\leq 8$ (see \cite{RS05,KN04}).}
\end{figure}

\begin{figure}
\begin{center}
\subfigure[$\CLE_3$ (from critical Ising model)]{\includegraphics[width=0.495\textwidth]{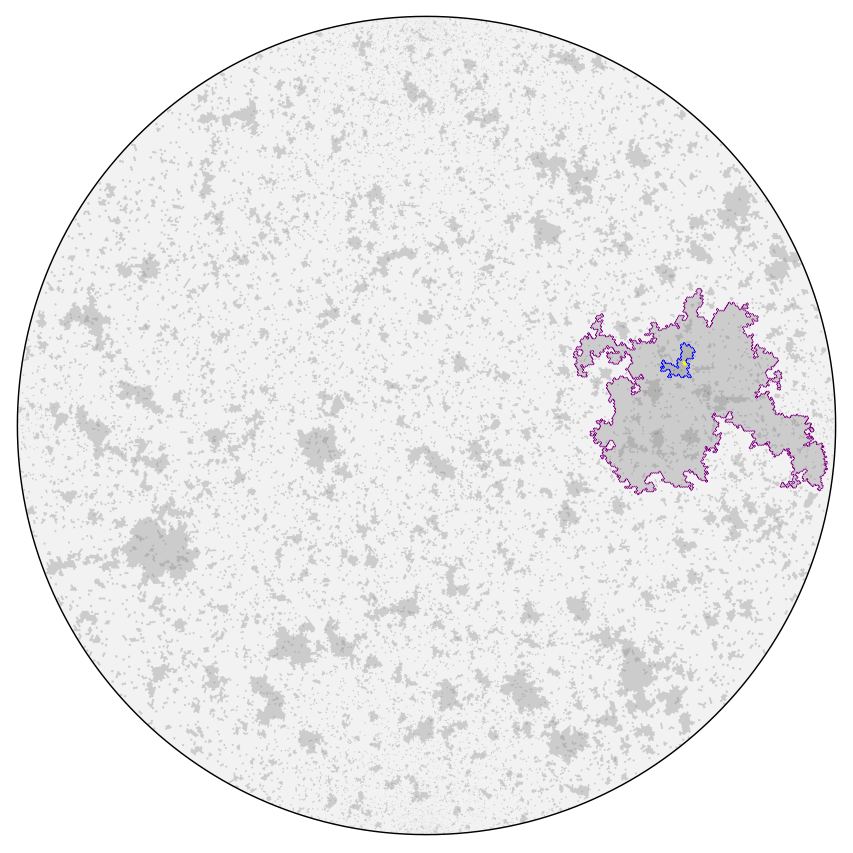}} \hfill
\subfigure[$\CLE_4$ (from the FK model with $q=4$) $\star$]{\includegraphics[width=0.495\textwidth]{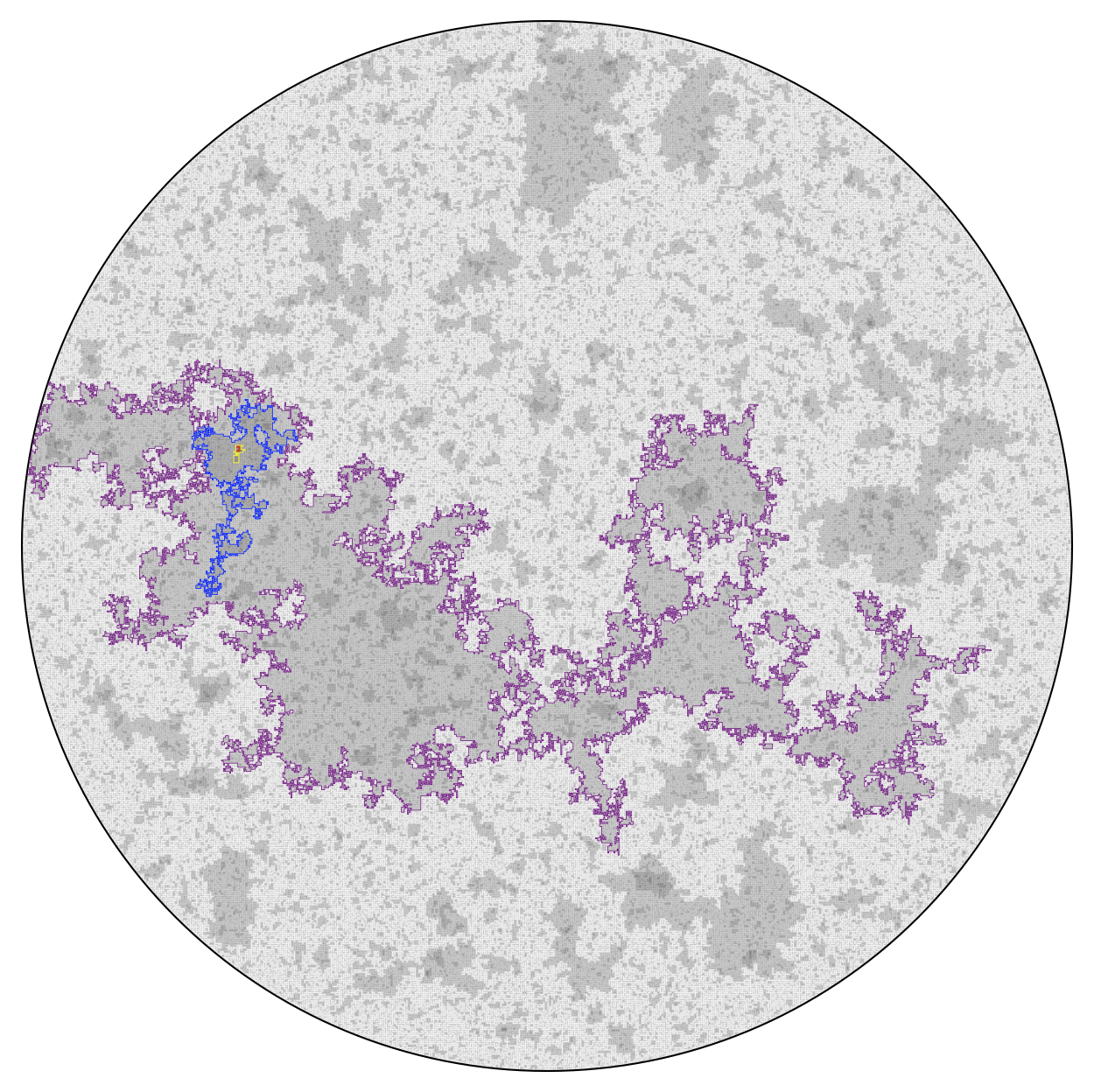}}\\[24pt]
\subfigure[$\CLE_{16/3}$ (from the FK model with $q=2$)]{\includegraphics[width=0.495\textwidth]{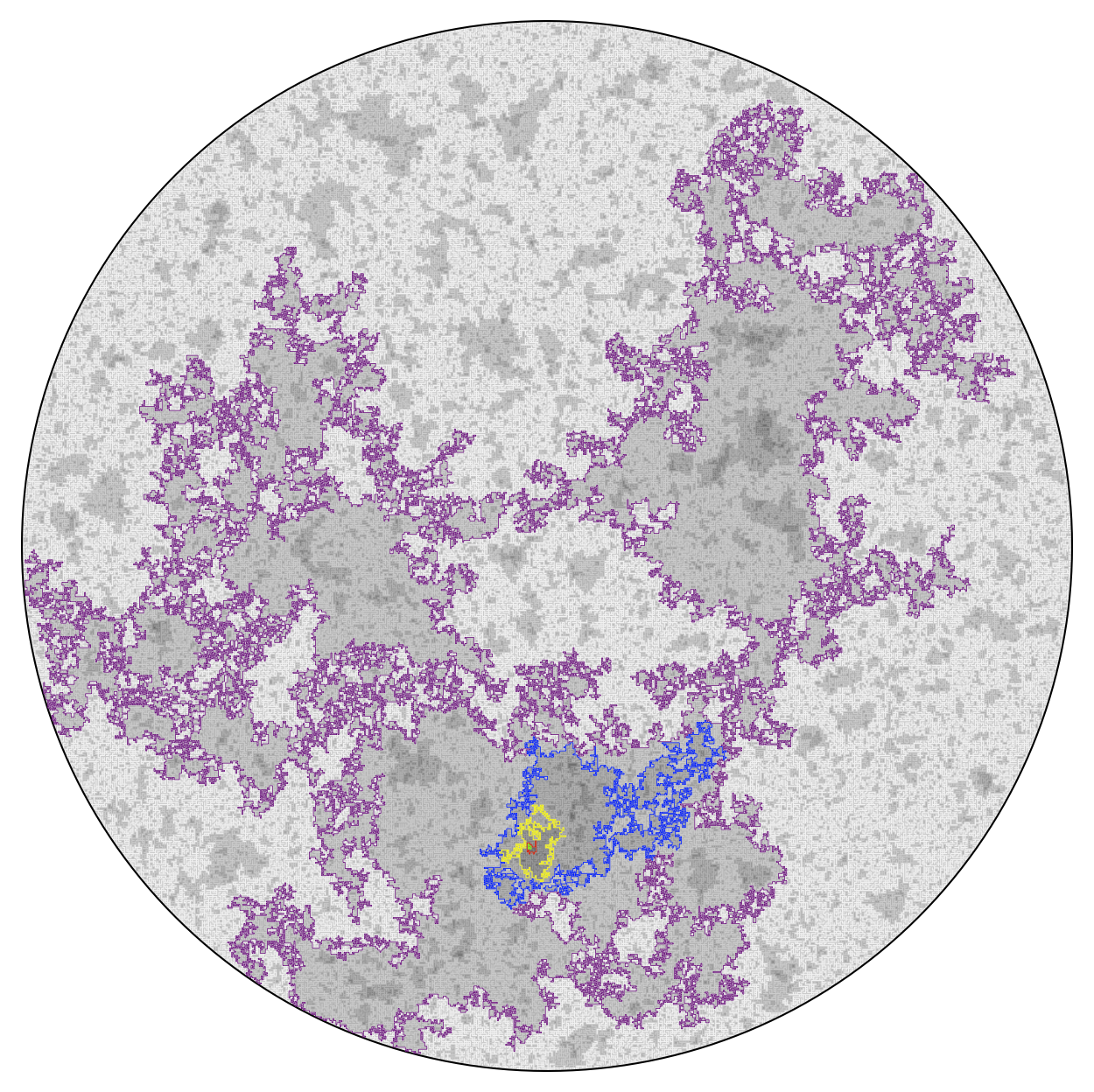}} \hfill
\subfigure[$\CLE_6$ (from critical bond percolation) $\star$]{\includegraphics[width=0.495\textwidth]{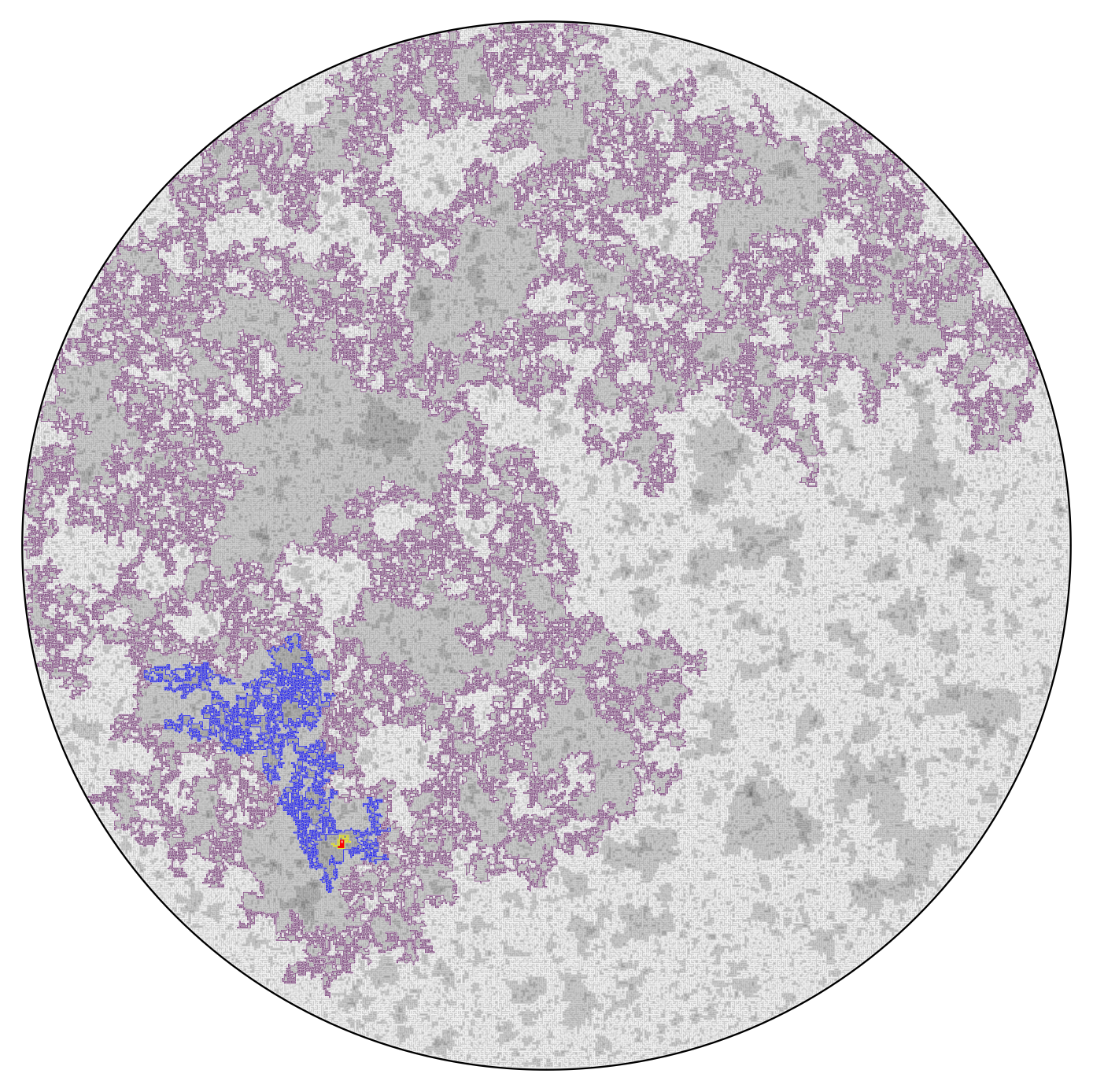}}
\end{center}
\caption{\label{fig::CLE_examples} Simulations of discrete loop
  models which converge to (or are believed to converge to, indicated with $\star$) $\CLE_\kappa$ in the fine mesh limit.
  For each of the $\CLE_\kappa$'s, one particular nested
  sequence of loops is outlined.  For $\CLE_\kappa$, almost all of the
  points in the domain are surrounded by an infinite nested sequence of
  loops, though the discrete samples shown here display only a few orders
  of nesting.}
\end{figure}

Let $\kappa \in (8/3,8)$, let $D \subsetneq \C$ be a simply connected
domain, and let $\Gamma$ \label{not::CLE} be a $\CLE_\kappa$ in $D$.  For
each point $z \in D$ and $\eps > 0$, we let
$\Loopcount_z(\eps)$ \label{not::Loopcount} be the number of loops of
$\Gamma$ which surround $B(z,\eps)$, the ball of radius $\eps$ centered at
$z$.  We prove the existence and conformal invariance of the limit as $\eps
\to 0$ of the random function $z\mapsto\Loopcount_z(\eps) -
\E[\Loopcount_z(\eps)]$ (with no additional normalization) in an
appropriate space of distributions (Theorem~\ref{thm::existence_weighted}).
We refer to this object as the \textit{nesting field\/} because, roughly,
its value describes the fluctuations of the nesting of $\Gamma$ around its
mean.  This result also holds when the loops are assigned i.i.d.\
weights. More precisely, we fix a probability measure $\mu$ on $\R$ with
finite second moment, define $\Gamma_z(\eps)$ \label{not::Gamma_z} to be
the set of loops in $\Gamma$ surrounding $B(z,\eps)$, and define
\begin{equation}
\label{eqn::normalized_loop_count}
\SLoopcount_z(\eps) =
\sum_{\Loop \in \Gamma_z(\eps)} \xi_\Loop\,,
\end{equation}
where $\xi_\Loop$ are i.i.d.\ random variables with law
$\mu$. \label{not::mu} We show that $z\mapsto\SLoopcount_z(\eps) -
\E[\SLoopcount_z(\eps)]$ converges as $\eps \to 0$ to a distribution we
call the \textit{weighted nesting field}. When $\kappa=4$ and $\mu$ is a
signed Bernoulli distribution, the weighted nesting field is the GFF
\cite{MS_CLE,extremes}.  Our result serves to generalize this construction
to other values of $\kappa \in (8/3,8)$ and weight measures $\mu$.  In
Theorem~\ref{thm::existence_weighted_alt}, we answer a question asked in
\cite[Problem~8.2]{SHE_CLE}.

The weighted nesting field is a random distribution, or generalized
function, on~$D$.  Informally, it is too rough to be defined pointwise
on~$D$, but it is still possible to integrate it against sufficiently
smooth compactly supported test functions on $D$.  More precisely, we prove
convergence to the nesting field in a certain local Sobolev space
$\Hloc^s(D)\subset C_c^\infty(D)'$ on $D$, where
$C_c^\infty(D)$ is the space of compactly supported smooth
  functions on $D$, $C_c^\infty(D)'$ is the space of distributions on
  $D$, and the index $s\in\R$ is a parameter characterizing how smooth
the test functions need to be. We review all the relevant definitions in
Section~\ref{sec::sobolev_spaces}.

The nesting field gives a loop-free description of the conformal loop ensemble.
For $\kappa\leq 4$ we believe that the nesting field determines the CLE, but that
for $\kappa>4$ the CLE contains more information.
(See Question~\hyperlink{ques::deterministic}{2} in the open problems section.)
In order to prove the existence of the nesting field, we show that the law
of CLE near a point rapidly forgets loops that are far away, in a sense that
we make quantitative.

Given $h\in C_c^\infty(D)'$ and $f\in C_c^\infty(D)$, we denote by
$\ang{h,f}$ the evaluation of the linear functional $h$ at $f$. Recall that
the pullback $h\circ \varphi^{-1}$ of $h \in C_c^\infty(D)'$ under a
conformal map $\varphi^{-1}$ is defined by $\langle
h\circ\varphi^{-1},f\rangle \colonequals \langle
h,|\varphi'|^2f\circ\varphi\rangle$ for $f\in C_c^\infty(\varphi(D))$.

\newpage
\begin{theorem}\label{thm::existence_weighted}
  Fix $\kappa\in(8/3,8)$ and $\delta>0$, and suppose $\mu$ is a probability
  measure on $\R$ with finite second moment.  Let $D \subsetneq \C$
  be a simply connected domain.  Let $\Gamma$ be a $\CLE_\kappa$ on $D$ and
  $(\xi_\CL)_{\CL\in\Gamma}$ be i.i.d.\ weights on the loops of $\Gamma$
  drawn from the distribution $\mu$.  Recall that for $\eps > 0$
  and $z \in D$, $\SLoopcount_z(\eps)$ denotes
  \[
  \SLoopcount_z(\eps) = \sum_{\substack{\CL\in\Gamma\\\text{$\CL$ surrounds
        $B(z,\eps)$}}} \xi_\CL\,.
  \]
  Let \label{not::h_eps}
  \begin{equation} h_\eps(z) = \SLoopcount_z(\eps) - \E[ \SLoopcount_z(\eps)]\,. \label{eq::h_eps}
  \end{equation}
  There exists an $\Hloc^{-2-\delta}(D)$-valued random variable
  $h=h(\Gamma,(\xi_\CL))$ such that for all $f\in C_c^\infty(D)$, almost
  surely $\lim_{\eps\to 0} \langle h_{\eps},f\rangle = \langle h,f\rangle$.
  Moreover, $h(\Gamma,(\xi_\CL))$ is
  almost surely a deterministic conformally invariant function of the CLE
  $\Gamma$ and the loop weights~$(\xi_\CL)_{\CL\in\Gamma}$: almost surely, for any
  conformal map $\varphi$ from $D$ to another simply connected
  domain, we have
  \[
  h(\varphi(\Gamma),(\xi_{\varphi(\CL)})_{\CL\in\Gamma}) =
  h(\Gamma,(\xi_\CL)_{\CL\in\Gamma}) \circ \varphi^{-1}\,.
  \]
\end{theorem}

In Theorem~\ref{thm::almost_sure_norm}, we prove a stronger form of
convergence, namely almost sure convergence in the norm topology of
$H^{-2-\delta}(D)$, when $\eps$ tends to $0$ along any given
geometric sequence.

We also consider the \textit{step nesting} sequence, defined by
\[
\mathfrak{h}_n(z) = \sum_{k=1}^n \xi_{\Loop_k(z)} - \E\left[\sum_{k=1}^n
  \xi_{\Loop_k(z)}\right], \quad n\in \N,
\]
where the random variables $(\xi_{\Loop})_{\Loop \in \Gamma}$ are
i.i.d.\ with law $\mu$. We may assume without loss of generality that $\mu$
has zero mean, so that $\mathfrak{h}_n(z) = \sum_{k=1}^n \xi_{\Loop_k(z)}$.
We establish the following convergence result for the step nesting
sequence, which parallels Theorem~\ref{thm::existence_weighted}:

\begin{theorem} \label{thm::existence_weighted_alt} Suppose that $D
  \subsetneq \C$ is a proper simply connected domain and $\delta>0$.
  Assume that the weight distribution $\mu$ has a finite second moment and
  zero mean. There exists an $\Hloc^{-2-\delta}(D)$-valued random variable
  $\mathfrak{h}$ such that $\lim_{n\to\infty} \mathfrak{h}_{n} =
  \mathfrak{h}$ almost surely in $\Hloc^{-2-\delta}(D)$.  Moreover,
  $\mathfrak{h}$ is almost surely determined by $\Gamma$ and
  $(\xi_\CL)_{\Loop\in\Gamma}$.

  Suppose that $\acute{D}$ is another simply connected domain and $\varphi
  \colon D \to \acute{D}$ is a conformal map.  Let $\acute{h}$ be the
  random element of $\Hloc^{-2-\delta}(\acute{D})$ associated with
  the $\CLE$ $\acute{\Gamma} = \varphi(\Gamma)$ on $\acute{D}$ and weights
  $(\xi_{\varphi^{-1}(\acute{\CL})})_{\acute{\CL} \in \acute{\Gamma}}$.
  Then
  $\acute{\mathfrak{h}} = \mathfrak{h} \circ \varphi^{-1}$ almost surely.
\end{theorem}

In Proposition~\ref{prop::fieldsequal}, we show that the step nesting field
and the weighted nesting field are equal, under the assumption that $\mu$
has zero mean.

When $\kappa=4$, $\gffparam =\sqrt{\pi/2}$, and $\mu =\mu_B$ where
$\mu_{\rm B}(\{\gffparam\}) = \mu_{\rm B}(\{-\gffparam\})=1/2$ (as in
Theorem~\ref{I-thm::gff_maximum} of \cite{extremes}) the distribution $h$ of
Theorem~\ref{thm::existence_weighted} is that of a GFF on $D$
\cite{MS_CLE}.  The existence of the distributional limit for other values
of $\kappa$ was posed in \cite[Problem~8.2]{SHE_CLE}.  Note that in this
context, $\tfrac{2}{\pi} \E[ \SLoopcount_z(\eps) \SLoopcount_w(\eps)]$ is
equal to the expected number of loops which surround both $B(z,\eps)$ and
$B(w,\eps)$. Let $G_D(z,w)$ \label{not::greens} be the Green's function for
the negative Dirichlet Laplacian on $D$. Since $\SLoopcount_z(\eps)$
converges to the GFF \cite{MS_CLE}, it follows that $\tfrac{2}{\pi} \E[
\SLoopcount_z(\eps) \SLoopcount_w(\eps)]$ converges to
$\frac{2}{\pi}G_D(z,w)$ (see Section 2 in \cite{DPRZ}).  That is, the
expected number of $\CLE_4$ loops which surround both $z$ and $w$ is given
by $\frac{2}{\pi}G_D(z,w)$.

One of the elements of the proof of Theorem~\ref{thm::existence_weighted}
is an extension of this bound which holds for all $\kappa \in (8/3,8)$.  We
include this as our final main theorem.

\begin{theorem}
\label{thm::surround_two_point_bound}
Let $\Gamma$ be a $\CLE_\kappa$ (with $8/3<\kappa<8$) on a simply connected
proper domain $D$.  For $z,w \in D$ distinct, let $\Loopcount_{z,w}$
\label{not::Loopcount_zw} be the number of loops of $\Gamma$ which surround
both $z$ and $w$.  For each integer $j\geq 1$, there exists a constant
$C_{\kappa,j} \in (0,\infty)$ such that
\begin{equation}
\label{eqn::expected_loops_bound}
\big|\E[\Loopcount_{z,w}^j] - (\lptyp \,2\pi\, G_D(z,w))^j\big| \leq C_{\kappa,j} (G_D(z,w)+1)^{j-1}\,.
\end{equation}
\end{theorem}

\subsection*{Outline}

 In Section~\ref{sec::cle_estimates} we review background material and
  establish some general CLE estimates, and in Section~\ref{sec::tail} we
  prove Theorem~\ref{thm::surround_two_point_bound}.
  Section~\ref{sec::field_regularity} includes proofs of several technical
  results used in the proof of Theorem~\ref{thm::existence_weighted}. In
  Section~\ref{sec::sobolev_spaces} we provide a brief overview of the
  necessary material on distributions and Sobolev spaces, and we establish
  a general result (Proposition~\ref{prop::convergence_Hminusd}) regarding
  the almost-sure convergence of a sequence of random distributions. In
  Sections~\ref{sec::weighted_loop_distribution} and~\ref{sec::step} we
  prove Theorems~\ref{thm::existence_weighted}
  and~\ref{thm::existence_weighted_alt}, respectively. We conclude by
  listing open questions in Section~\ref{sec::questions}.

\section{Basic CLE estimates}
\label{sec::cle_estimates}

\makeatletter{}
In this section we record some facts about CLE. We refer the reader to
  the preliminaries section in \cite{extremes} for an
  introduction to CLE. We begin by reminding the reader of the Koebe
distortion theorem and the Koebe quarter theorem.
\begin{theorem} (Koebe distortion theorem) \label{thm::distortion} If $f:\D
  \to \C$ is an injective analytic function and $f(0)=0$, then
\[
\frac{r}{(1+r)^2}|f'(0)| \leq |f(re^{i\theta})| \leq
\frac{r}{(1-r)^2}|f'(0)|, \quad \text{for }\theta \in \R \text{ and }  0 \leq r < 1\,.
\]
\end{theorem}
The Koebe quarter theorem, which says that $B(0,\tfrac{1}{4}|f'(0)|)
\subset f(\D)$, follows from the lower bound in the distortion theorem
\cite[Theorem~3.17]{LAW05}.  Combining the quarter theorem with the Schwarz
lemma \cite[Lemma~2.1]{LAW05}, we obtain the following corollary.
\begin{corollary} \label{cor::onefourth} If
  $D\subsetneq \C$ is a simply connected domain, $z\in D$, and $f:\D \to D$
  is a conformal map sending 0 to $z$, then the inradius
  $\inrad(z;D) \colonequals \inf_{w \in \C \setminus D}|z-w|$
  and the conformal radius $\confrad(z;D)\colonequals |f'(0)|$ satisfy
\[
\inrad(z;D) \leq \confrad(z;D) \leq 4\,\inrad(z;D)\,.
\]
\end{corollary}

\label{not::loop} For the $\CLE_\kappa$ $\Gamma$ in $D$, $z\in D$, and $j
\geq 0$, we define $\Loop_z^j$ to be the $j$th outermost loop of $\Gamma$
which surrounds $z$. For $r>0$, we define
\begin{subequations}
\label{eq::Jcapsubset} \label{not::Jcapsubset}
\begin{align}
J^\cap_{z,r}&\colonequals\min\{j \geq 1 : \Loop_z^j \cap B(z,r) \neq \varnothing\} \\
J^\subset_{z,r}&\colonequals\min\{j \geq 1 : \Loop_z^j \subset B(z,r) \}.
\end{align}
\end{subequations}

\begin{lemma}
\label{lem::loop_contain_prop}
For each $\kappa \in (8/3,8)$ there exists $p = p(\kappa) > 0$ such that
for any proper simply connected domain~$D$ and $z \in D$,
\[ \P[\Loop_z^2 \subseteq B(z,\dist(z,\partial D))] \geq p.\]
\end{lemma}

\begin{corollary}
  \label{cor::loop_contain_stoch_dom}
  $J^\subset_{z,r} - J^\cap_{z,r}$ is
  stochastically dominated by $2\wt{N}$ where $\wt{N}$ is a geometric random
  variable with parameter $p = p(\kappa) > 0$ which depends only on
  $\kappa \in (8/3,8)$.
\end{corollary}

\begin{proof}
  See Corollary~\ref{I-cor::loop_contain_stoch_dom} in \cite{extremes}.
\end{proof}

We use the following estimate for the overshoot of a random walk the
  first time it crosses a given threshold. We will apply this lemma to the
  random walk which tracks the negative log conformal radius of the
  sequence of CLE loops surrounding a given point $z\in D$, as viewed from
  $z$. See Lemma~\ref{I-lem::firsthitting} in \cite{extremes} for a proof.

\begin{lemma}
\label{lem::firsthitting}
Suppose $\{X_j\}_{j\in \N}$ are nonnegative i.i.d.\
random variables for which $\E[X_1] > 0$ and $\E[e^{\mgfparam_0
 X_1}]<\infty$ for some $\mgfparam_0>0$.
 Let $S_n=\sum_{j=1}^n X_j$
 and $\tau_x =
\inf\{n \geq 0 : S_n \geq x\}$. Then there exists $C>0$ (depending on
the law of $X_1$ and $\lambda_0$) such that $\P[S_{\tau_{x}} - x \geq \alpha] \leq C\exp(-\mgfparam_0
\alpha)$ for all $x\geq 0$ and $\alpha>0$.
\end{lemma}

The following lemma provides a quantitative version of the statement that
it is unlikely that there exists a CLE loop surrounding the inner boundary
but not the outer boundary of a given small, thin annulus.
We make use of a quantitative coupling between CLE in large domains and
full-plane CLE, which appears as Theorem~\ref{thm::full-plane} in the appendix.

\begin{lemma} \label{lem::mean_loops_inr} Let $\Gamma$ be a $\CLE_\kappa$
  in $\D$. There exist constants $C>0$, $\alpha>0$, and $\eps_0>0$
  depending only on $\kappa$ such that for $0<\eps<\eps_0$ and $0\leq
  \delta < 1/2$,
  \begin{equation} \label{eq::suffices}
    \E[ \Loopcount_{0}(\eps(1-\delta)) - \Loopcount_{0}(\eps)] \leq C \delta + C \eps^\alpha\,.
  \end{equation}
\end{lemma}

\begin{proof}
   We couple the $\CLE_\kappa$ $\Gamma_\D=\Gamma$ in the disk with a
    whole-plane $\CLE_\kappa$ $\Gamma_\C$ as in
    Theorem~\ref{thm::full-plane}.  Index the loops of $\Gamma_\C$
    surrounding $0$ by $\Z$ in such a way that $\Loop_0^n(\Gamma_\C)$ and
    $\Loop_0^n(\Gamma_\D)$ are exponentially close for large $n$.  For
    $n\in\N$ define $V^\D_n=-\log\inrad \Loop_0^n(\Gamma_\D)$, and for
    $n\in\Z$ define $V^\C_n=-\log\inrad \Loop_0^n(\Gamma_\C)$.  Since
    whole-plane $\CLE_\kappa$ is scale invariant, the set
      $\{V^\C_n\,:n\in \Z\}$ is translation invariant.  Using
    Corollary~\ref{cor::onefourth} to compare $(V^\C_n)_{n\in \Z}$ to the
    sequence of log conformal radii of the loops of $\Gamma_\C$ surrounding
    the origin, the translation invariance implies
  \[ \E\left[\#\left\{n\,:\,a\leq V^\C_n <b\right\}\right] = \lptyp(b-a)\,. \]

  Let $\alpha$ and the term \textit{low distortion\/} be defined as in
    the statement of Theorem~\ref{thm::full-plane}.  With probability
  $1-O(\eps^\alpha)$ there is a low distortion map from
  $\Gamma_\D|_{B(0,\eps)^+}$ to $\Gamma_\C|_{B(0,\eps)^+}$, and on this
  event, we can bound
 \begin{multline*}
 \# \left\{n\,:\,\log\frac{1}{\eps}\leq V^\D_n < \log\frac{1}{\eps(1-\delta)}\right\} \\
  \leq \#\left\{n\,:\, \log\frac{1}{\eps}-O(\eps^\alpha) \leq V^\C_n < \log\frac{1}{\eps(1-\delta)} +O(\eps^\alpha)\right\}\,.
\end{multline*}

On the event that there is no such low distortion map, this can be detected
by comparing the boundaries of $\Gamma_\D|_{B(0,\eps)^+}$ and
$\Gamma_\C|_{B(0,\eps)^+}$, so that conditional on this unlikely event,
$\Gamma_\D|_{B(0,\eps)^+}$ is still an unbiased $\CLE_\kappa$ conformally
mapped to the region surrounded by the boundary of
$\Gamma_\D|_{B(0,\eps)^+}$.  In particular, the sequence of log-conformal
radii of loops of $\Gamma_\D|_{B(0,\eps)^+}$ surrounding $0$ is a renewal
process, which together with the \hyperref[thm::distortion]{Koebe distortion theorem}
and the bound $\delta\leq 1/2$ imply
  \[ \E[ \Loopcount_{0}(\eps(1-\delta)) - \Loopcount_{0}(\eps) \,|\,
  \text{no low distortion map}] \leq \text{constant}\,. \]
Combining these bounds yields \eqref{eq::suffices}.
\end{proof}

\begin{lemma}
\label{lem::mean_loops_lcr}
\label{lem::small_second_moment}
For each $\kappa\in(8/3,8)$ and integer $j\in\N$, there are constants
$C>0$, $\alpha>0$, and $\eps_0>0$ (depending only on $\kappa$ and $j$)
such that whenever $D$ is a simply connected proper domain, $z\in D$,
$\varphi$ is a conformal transformation of $D$,
and $0<\eps<\eps_0$, if $\Gamma$ is a $\CLE_\kappa$ in $D$, then
\[
\E\bigg[ \big|\Loopcount_z(\eps \confrad(z;D) ;\Gamma) - \Loopcount_{\varphi(z)}(\eps \confrad(\varphi(z);\varphi(D));\varphi(\Gamma)) \big|^j\bigg] \leq C \eps^\alpha\,.
\]
\end{lemma}

\begin{proof}
Observe that translating and scaling the domain $D$ or its conformal
  image $\varphi(D)$ has no effect on the loop counts, so we assume
  without loss of generality that $z=0$, $\varphi(z)=0$,
  $\confrad(z;D)=1$, and $\confrad(\varphi(z);\varphi(D))=1$.  Observe
  also that it suffices to prove this lemma in the case that the
  domain $D$ is the unit disk $\D$, since a general $\varphi$ may be
  expressed as the composition $\varphi = \varphi_2 \circ
  \varphi_1^{-1}$ where $\varphi_1$ and $\varphi_2$ are conformal
  transformations of the unit disk with $\varphi_i(0)=0$ and
  $\varphi_i'(0)=1$, and the desired bound follows from the triangle
  inequality.

  Let $\Gamma$ be a $\CLE_\kappa$ on $\D$, and let $\acute\Gamma=\varphi(\Gamma)$.
  By the
  \hyperref[thm::distortion]{Koebe distortion theorem} and the
  elementary inequality
  \begin{equation} \label{eq::koebe_estimate}
    1-3r\leq\frac{1}{(1+r)^2}\leq\frac{1}{(1-r)^2}\leq 1+3r,
  \quad\text{for $r$ small enough},
  \end{equation}
  we have
  \[
  B(0,\eps-3\eps^2)\subset
  \varphi^{-1}(B(0,\eps)) \subset B(0,\eps+3\eps^2)\,,
  \]
  for small enough $\eps$.  Hence $
  \Loopcount_0 (\eps+3\eps^2; \Gamma) \leq \Loopcount_0 (\eps;
  \acute{\Gamma}) \leq \Loopcount_0 (\eps-3\eps^2; \Gamma)$,
  and so for
  \[ X\colonequals \Loopcount_0(\eps-3\eps^2;\Gamma) - \Loopcount_0(\eps+3\eps^2;\Gamma) \]
  we have
  $|\Loopcount_0(\eps;\acute{\Gamma}) - \Loopcount_0(\eps;\Gamma)| \leq X$.

  By Lemma~\ref{lem::mean_loops_inr} we have $\E[X]=O(\eps^\alpha)$, which proves the case $j=1$.

  Notice that the conformal radius of
  every new loop after the first that intersects $B(0,\eps+3\eps^2)$ has a
  uniformly positive probability of being less than
  $\frac{1}{4}(\eps-3\eps^2)$, conditioned on the previous loop.  By the
  \hyperref[cor::onefourth]{Koebe quarter theorem}, such a loop
  intersects $B(0,\eps-3\eps^2)$.
  Thus for some $p<1$ we have $\P[X \geq k+1] \leq p \P[X \geq k]$ for $k\geq 0$.
  Hence
  \begin{align*}
 \E[X^j] = \sum_{k=1}^\infty k^j \P[X=k]
  &\leq \sum_{k=1}^{\infty}k^j p^k\P[X=1]
  \leq \left(\sum_{k=1}^{\infty}k^j p^k\right) \E[X]=O(\eps^\alpha)\,,
  \end{align*}
  which proves the cases $j>1$.
\end{proof}

\section{Co-nesting estimates}
\label{sec::tail}

We use the following lemma in the proof of Theorem~\ref{thm::surround_two_point_bound}:

\begin{lemma}
\label{lem::firsthitting_mg}
Let $\mgfparam_0>0$, and suppose $\{X_j\}_{j\in \N}$ are nonnegative
i.i.d.\ random variables for which $\E[X_1] > 0$ and $\E[e^{\mgfparam_0
  X_1}]<\infty$. Let $\Lambda(\mgfparam) = \log \E[ e^{\mgfparam X_1}]$ and
let $S_n=\sum_{j=1}^n X_j$.  For $x > 0$, define $\tau_x = \inf\{n \geq 0 :
S_n \geq x\}$.  For $\mgfparam <\mgfparam_0$, let
\[ M_n^\mgfparam = \exp(\lambda S_n - \Lambda(\lambda) n).\]
Then for $\mgfparam<\mgfparam_0$ and $x\geq0$, the random variables $\{M_{n \wedge \tau_x}^{\mgfparam}\}_{n\in \N}$ are uniformly integrable.
\end{lemma}
\begin{proof}
  Fix $\beta > 1$ such that $\beta \mgfparam < \lambda_0$.  By H\"older's
  inequality, any family of random variables which is uniformly bounded in
  $L^p$ for some $p>1$ is uniformly integrable.
    Therefore, it suffices to show that $\sup_{n \geq 0}\E[ (M_{n \wedge
    \tau_x}^\mgfparam)^\beta] < \infty$.  We have,
\begin{align*}
     (M_{n \wedge \tau_x}^\mgfparam)^\beta
&= \exp( \beta \lambda (S_{n \wedge \tau_x} - x)) \times \exp( \beta \mgfparam x - \beta \Lambda(\mgfparam) (n \wedge \tau_x ))\\
&\leq \exp( \beta \lambda (S_{\tau_x} - x))  \times \exp( \beta \mgfparam x).
\end{align*}
The result follows from Lemma~\ref{lem::firsthitting}.
\end{proof}

\begin{proof}[Proof of Theorem~\ref{thm::surround_two_point_bound}]
Fix $z,w \in D$ distinct and $j \in \N$.  Let $\varphi \colon D \to \D$ be the conformal map which sends $z$ to $0$ and $w$ to $e^{-x} \in (0,1)$.  Let $G_D$ (resp.\ $G_\D$) be the Green's function for $-\Delta$ with Dirichlet boundary conditions on~$D$ (resp.\ $\D$).  Explicitly,
\[ G_\D(u,v) = \frac{1}{2\pi} \log\frac{|1-\ol{u} v|}{|u-v|} \quad\text{for}\quad u,v \in \D.\]
In particular, $G_\D(0,u) = \frac{1}{2\pi}\log|u|^{-1}$ for $u \in \D$.  By the conformal invariance of $\CLE_\kappa$ and the Green's function, i.e.\ $G_D(u,v) = G_\D(\varphi(u),\varphi(v))$, it suffices to show that there exists a constant $C_{j,\kappa} \in (0,\infty)$ which depends only on $j$ and $\kappa \in (8/3,8)$ such that
\begin{align}
 \big| \E[ (\Loopcount_{0,e^{-x}})^j] -  (\lptyp x )^j \big|
\leq  C_{j,\kappa} (x+1)^{j-1} \quad\quad\text{for all $x>0$}. \label{eqn::sufficient_estimate}
\end{align}

Let $\{T_i\}_{i \in \N}$ be the sequence of $\log$ conformal radii
increments associated with the loops of $\Gamma$ which surround $0$, let
$S_k = \sum_{i=1}^k T_i$, and let $\tau_x = \min\{k \geq 1 : S_k \geq
x\}$. Recall that $\Lambda_\kappa(\lambda)$ denotes the $\log$ moment
generating function of the law of $T_1$. Let $M_n = \exp(\lambda S_n -
\Lambda_\kappa(\lambda) n)$.  By Lemma~\ref{lem::firsthitting_mg}, $\{M_{n
  \wedge \tau_x}\}_{n\in \N}$ is a uniformly integrable martingale for
$\lambda < 1-\tfrac{2}{\kappa} - \tfrac{3\kappa}{32}$.
By Lemma~\ref{lem::firsthitting}, we
can write $S_{\tau_x} = x + X$ where $\E[e^{\lambda X}] < \infty$.
By the optional stopping theorem for uniformly integrable martingales
  (see \cite[\S~A14.3]{williams1991probability}), we have that
\begin{equation}\label{eqn::exp_moment}
  1 = \E[ \exp(\lambda S_{\tau_x} - \Lambda_\kappa(\lambda) \tau_x)] = \E[ \exp(\lambda x + \lambda X - \Lambda_\kappa(\lambda) \tau_x)].
\end{equation}

We argue by induction on $j$ that
\begin{equation}\label{eqn::tau_moment}
\E[(\Lambda'_\kappa(0) \tau_x)^j] = x^j + O((x+1)^{j-1}).
\end{equation}
The base case $j=0$ is trivial.

If we differentiate \eqref{eqn::exp_moment} with respect to $\lambda$ and then evaluate at $\lambda=0$, we obtain
\[0 = \E[(x + X - \Lambda'_\kappa(0) \tau_x)].\]
If we instead differentiate twice, we obtain
\[0 = \E[(x + X - \Lambda'_\kappa(0) \tau_x)^2 -\Lambda''_\kappa(0)\tau_x].\]
Similarly, if we differentiate $j$ times with respect to $\lambda$ and then evaluate at $\lambda=0$, we obtain
\begin{equation}
 0 = \E[(x + X - \Lambda'_\kappa(0) \tau_x)^j] + \sum_{\substack{i\geq0,k\geq 1\\i+2k\leq j}} A_{\kappa,i,k} \E[(x + X - \Lambda'_\kappa(0) \tau_x)^i \tau_x^k],
\end{equation}
where the $A_{\kappa,i,k}$'s are constant coefficients depending on the higher order derivatives of $\Lambda_\kappa$ at $0$.
By our induction hypothesis, for $h<j$ we have $\E[\tau_x^{h}]=O((x+1)^{h})$.  Conditional on $\tau_x$, $X$ has exponentially small tails, so $\E[\tau_x^{h} X^{\ell}]=O((x+1)^{h})$ as well.  From this we obtain
\begin{equation}
 0 = \E[(x - \Lambda'_\kappa(0) \tau_x)^j] + O((x+1)^{j-1}).
\end{equation}
Using our induction hypothesis again for $h<j$, we obtain
\begin{equation}
 0 = \sum_{h=0}^{j-1} \binom{j}{h} (-1)^{h} x^j + \E[(-\Lambda'_\kappa(0) \tau_x)^j] + O((x+1)^{j-1}),
\end{equation}
from which \eqref{eqn::tau_moment} follows, completing the induction.

Recall that $J_{0,r}^{\cap}$ (resp.\ $J_{0,r}^{\subset}$) is the smallest
index $j$ such that $\Loop_0^j$ intersects (resp.\ is contained in)
$B(0,r)$.  It is straightforward that \[ \tau_{x-\log4} \leq
J_{0,e^{-x}}^{\cap} \leq \Loopcount_{0,e^{-x}}+1 \leq
J_{0,e^{-x}}^{\subset}.\] Since the $\tau$'s are stopping times for an
i.i.d.\ sum, conditional on the value of $\tau_{x-\log 4}$, the difference
$\tau_x-\tau_{x-\log 4}$ has exponentially decaying tails.  Moreover, by
Lemma~\ref{lem::loop_contain_prop}, conditional on the value of $\tau_x$,
$J_{0,e^{-x}}^{\subset}-\tau_x$ has exponentially decaying tails.  Thus
$\E[\Loopcount_{0,e^{-x}}^j] = \E[\tau_x^j] + O((x+1)^{j-1})$.  Finally, we
recall that $1/\Lambda_\kappa'(0) = 1/\E[T_1] = \lptyp$.
\end{proof}

By combining Theorem~\ref{thm::surround_two_point_bound} and
Corollary~\ref{cor::loop_contain_stoch_dom}, we can estimate the moments of
the number of loops which surround a ball in terms of powers of
$G_D(z,w)$.
\begin{corollary}
\label{cor::two_point_moments}
There exists a constant $C_{j,\kappa} \in (0,\infty)$ depending only on $\kappa \in (8/3,8)$ and $j\in\N$ such that the following is true.
For each $\eps > 0$ and $z\in D$ for which $\dist(z,\partial D) \geq 2\eps$ and $\theta \in \R$, we have
\begin{equation}
\label{eqn::contain_moment_green_bound}
 \big| \E[ (\Loopcount_z(\eps))^j] - (2\pi \lptyp G_D(z,z+\eps e^{i\theta}))^j\big|
 \leq C_{j,\kappa}(G_D(z,z+\eps e^{i\theta})+1)^{j-1}.
\end{equation}
In particular, there exists constant a constant $C_\kappa \in (0,\infty)$ depending only on $\kappa \in (8/3,8)$ such that
\begin{equation}
\label{eqn::contain_mean_cr_bound}
 \left|\E[ \Loopcount_z(\eps)] - \lptyp \log \frac{\confrad(z;D)}{\eps}\right| \leq C_\kappa\,.
\end{equation}
\end{corollary}
\begin{proof}
Let $w=z+\eps e^{i\theta}$.
Corollary~\ref{cor::loop_contain_stoch_dom} implies that $|\Loopcount_{z,w} - \Loopcount_z(\eps)|$ is stochastically dominated by a geometric random variable whose parameter $p$ depends only on~$\kappa$.  Consequently, \eqref{eqn::contain_moment_green_bound} is a consequence of Theorem~\ref{thm::surround_two_point_bound}.  To see \eqref{eqn::contain_mean_cr_bound}, we apply \eqref{eqn::contain_moment_green_bound} for $j=1$ and use that $G_D(u,v) = \tfrac{1}{2\pi} \log |u-v|^{-1} - \psi_u(v)$ where $\psi_u(v)$ is the harmonic extension of $v \mapsto \tfrac{1}{2\pi} \log |u-v|^{-1}$ from $\partial D$ to $D$.  In particular, $\psi_z(z) = \tfrac{1}{2\pi} \log \confrad(z;D)$.
\end{proof}

\section{\texorpdfstring{Regularity of the $\eps$-ball nesting field}{Regularity of the \textepsilon-ball nesting field}}
\label{sec::field_regularity}

\makeatletter{}A key estimate that we use in the proof of Theorem~\ref{thm::existence_weighted} is the following
bound on how much the centered nesting field $h_\eps$ depends on $\eps$.  The proof of Theorem~\ref{thm::var-field-diff} and
the remaining sections may be read in either order.

\begin{theorem} \label{thm::var-field-diff}
  Let $D$ be a proper simply connected domain, and let $h_\eps(z)$ be
  the centered weighted nesting around the ball $B(z,\eps)$ of a
  $\CLE_\kappa$ on $D$, defined in \eqref{eq::h_eps}.  Suppose
  $0<\eps_1(z)\leq\eps$ and $0<\eps_2(z)\leq\eps$ on a compact subset
  $K\subset D$ of the domain.  Then there is some $c>0$ (depending on $\kappa$)
  and $C_0>0$ (depending on $\kappa$, $D$, $K$, and the loop weight distribution)
  for which
\begin{equation} \label{eq::var-field-diff}
\iint\limits_{K\times K} \big|\E\big[
   (h_{\eps_1(z)}(z){-}h_{\eps_2(z)}(z))\,(h_{\eps_1(w)}(w){-}h_{\eps_2(w)}(w))
  \big]\big| \,dz\,dw \leq C_0 \eps^c\,.
\end{equation}
\end{theorem}

\begin{proof}
  Let $A$, $B$, and $C$ be the disjoint sets of loops for which
    $A\cup B$ is the set of loops surrounding $B(z,\eps_1(z))$ or
    $B(z,\eps_2(z))$ but not both, and $B\cup C$ is the set of loops
    surrounding $B(w,\eps_1(w))$ or $B(w,\eps_2(w))$ but not both.
  Letting $\xi_\Loop$ denote the weight of loop $\Loop$, then we have
\begin{align} \notag
\E[ (h_{\eps_1(z)}&(z) {-} h_{\eps_2(z)}(z))(h_{\eps_1(w)}(w) {-} h_{\eps_2(w)}(w))] \\ \notag
  &= \cov[h_{\eps_1(z)}(z) {-} h_{\eps_2(z)}(z),h_{\eps_1(w)}(w) - h_{\eps_2(w)}(w)] \\ \label{eq::h_diff}
  &= \pm\cov\left[\sum_{a\in A} \xi_a+\sum_{b\in B} \xi_b,
                  \sum_{b\in B} \xi_b+\sum_{c\in C} \xi_c\right] \\ \notag
  &= \pm \Var[\xi]\,\E[|B|] \, \pm \E[\xi]^2\Cov[|A|{+}|B|,|B|{+}|C|]\\ \notag
  &= \pm \Var[\xi]\,\E[|B|] \, + \E[\xi]^2
      \cov(\Loopcount_z(\eps_1){-}\Loopcount_z(\eps_2),
           \Loopcount_w(\eps_1){-}\Loopcount_w(\eps_2))\,, \notag
\end{align}
where the $\pm$ signs are the sign of $(\eps_1(z){-}\eps_2(z))(\eps_1(w){-}\eps_2(w))$.

Let $G_D^{\kappa,\eps}(z,w)$ denote the expected number of loops surrounding
$z$ and $w$ but surrounding neither $B(z,\eps)$ nor $B(w,\eps)$.
\label{not::GDke}
Then $\E[|B|]\leq G_D^{\kappa,\eps}(z,w)$.
In Lemma~\ref{lem::mean_bound} we prove
\begin{equation*}
  \iint_{K\times K} G_D^{\kappa,\eps}(z,w) \,dz\,dw \leq C_1 \eps^c\,,
\end{equation*}
and in Lemma~\ref{lem::cov_bound} we prove
\begin{equation*}
  \iint_{K\times K} \big|\cov(\Loopcount_z(\eps_1(z)) - \Loopcount_z(\eps_2(z)),
  \Loopcount_w(\eps_1(w)) - \Loopcount_w(\eps_2(w)))\big| \, dz\,dw \leq C_2 \eps^c\,,
\end{equation*}
where $c$ depends only on $\kappa$ and $C_1$ and $C_2$ depend only on $\kappa$, $D$, and $K$.
Equation~\eqref{eq::var-field-diff} follows from these bounds.
\end{proof}

In the remainder of this section we prove Lemmas~\ref{lem::mean_bound} and~\ref{lem::cov_bound}.

\begin{lemma} \label{lem::near_far} For any $\kappa\in(8/3,8)$ and $j\in
    \N$, there is a positive constant $c>0$ such that, whenever
    $D\subsetneq \C$ is a simply connected proper domain, $z\in D$, and
    $0<\eps<r$, the $j^{\text{th}}$ moment of the number of $\CLE_\kappa$
    loops surrounding $z$ which intersect $B(z,\eps)$ but are not contained
    in $B(z,r)$ is $O((\eps/r)^c)$.
\end{lemma}
\begin{proof}
  If there is a loop $\Loop=\Loop_z^k$ surrounding $z$ which is not
  contained in $B(z,r)$ and comes within distance $\eps$ of $z$, then
  $J_{z,\eps}^\cap\leq k$ and $J_{z,r}^\subset> k$, so $J_{z,\eps}^\cap <
  J_{z,r}^\subset$.  But from Corollary~\ref{cor::loop_contain_stoch_dom}
  $J_{z,r}^\subset-J_{z,r}^\cap$ is dominated by twice a geometric random
  variable, and by Lemma~\ref{I-lem::firsthitting} in \cite{extremes}
  together with the Koebe quarter theorem we have
  $J_{z,\eps}^\cap-J_{z,r}^\cap$ is order $\log(r/\eps)$ except with
  probability $O((\eps/r)^{c_1})$, for some constant $c_1>0$ (depending on
  $\kappa$).  Therefore, except with probability $O((\eps/r)^{c_2})$ (with
  $c_2=c_2(\kappa)>0$), we have $J_{z,\eps}^\cap \geq J_{z,r}^\subset$. In
  this case there is no loop $\Loop$ surrounding $z$, not contained in
  $B(z,r)$, and coming within distance $\eps$ of $z$.  Finally, note that
  conditioned on the event that there is such a loop $\Loop$, the
  conditional expected number of such loops is by
  Corollary~\ref{cor::loop_contain_stoch_dom} dominated by twice a
  geometric random variable.
\end{proof}

\begin{lemma} \label{lem::mean_bound} For some positive constant $c<2$,
  \begin{equation}
  \label{eq::mean_bound}
  \iint_{K \times K} G_D^{\kappa,\eps}(z,w) \,dz \,dw=O(\area(K)^{2-c/2} \eps^c)\,.
\end{equation}
\end{lemma}

\begin{proof}
  Let $F^\eps_{z,w}$ denote the number of
  loops surrounding both $z$ and $w$ but not $B(z,\eps)$ or $B(w,\eps)$.
  Then $G_D^{\kappa,\eps}(z,w)=\E[F^\eps_{z,w}]$.

  Suppose $|z-w|\leq\eps$.  Let $\Loop$ be the outermost loop (if any)
  surrounding both $z$ and $w$ but not $B(z,\eps)$ or $B(w,\eps)$.  The
  number of additional such loops is $\Loopcount_{z,w}(\Gamma')$, where
  $\Gamma'$ is a $\CLE_\kappa$ in $\interior\Loop$, and by
  Theorem~\ref{thm::surround_two_point_bound} we have
  $\E[\Loopcount_{z,w}(\Gamma')]\leq C_1 \log(\eps/|z-w|) + C_2$ for some
  constants $C_1$ and $C_2$.  Integrating the logarithm, we find that
  \begin{equation} \label{eq::diagonal}
  \iint_{\substack{K\times K\\|z-w|\leq\eps}} G_D^{\kappa,\eps}(z,w) \,dz\,dw = O(\area(K) \eps^2)\,.
  \end{equation}

  Next suppose $|z-w|>\eps$.  Now $F^\eps_{z,w}$ is dominated by the number of loops surrounding $z$ which intersect $B(z,\eps)$ but are not
  contained in $B(z,|z-w|)$, and Lemma~\ref{lem::near_far} bounds the expected number of these loops by $O((\eps/|z-w|)^c)$ for some $c>0$.
  We decrease $c$ if necessary to ensure $0<c<2$, and let $R=\area(K)^{1/2}$.
  Since $(\eps/|z-w|)^c$ is decreasing in $|z-w|$, we can bound
  \begin{align}
  \iint_{\substack{K\times K\\|z-w|>\eps}} G_D^{\kappa,\eps}(z,w) \,dz\,dw
   &\leq
  \iint_{\substack{R\D\times R\D\\|z-w|>\eps}} O((\eps/|z-w|)^c) \,dz\,dw \notag\\
  &= O(\area(K)^{2-c/2} \eps^{c})\,.\label{eq::off-diagonal}
  \end{align}
  Combining \eqref{eq::diagonal} and \eqref{eq::off-diagonal}, using again $c<2$, we obtain \eqref{eq::mean_bound}.
\end{proof}

We let $S_{z,w}$ be the index of the outermost loop surrounding $z$ which
separates $z$ from $w$ in the sense that $w\notin U_z^{S_{z,w}}$.  Note
that $S_{z,w}$ is also the smallest index for which $z\notin
U_w^{S_{z,w}}$:
\begin{equation}
 S_{z,w} \colonequals \min\{k:w\notin U_z^k\} = \min\{k:z\notin U_w^k\} \,.
\end{equation}
We let $\Sigma_{z,w}$ denote the $\sigma$-algebra
\begin{equation} \label{eq::sigma}
  \Sigma_{z,w} \colonequals \sigma(\{\Loop_z^{k}\,:1\leq k \leq S_{z,w}\}\cup \{\Loop_w^{k}\,:1\leq k \leq S_{z,w}\})\,.
\end{equation}

\begin{lemma}
  There is a constant $C$ (depending only on $\kappa$) such that
  if $z,w\in D$ are distinct, then
\[
  -C\leq \E\!\left[\log\frac{\confrad(z;U_z^{S_{z,w}})}{\min(|z-w|,\confrad(z;D))}\right] \leq C\,.
\]
\end{lemma}
\begin{proof}
  Let $r=\min(|z-w|,\dist(z,\partial D))$.
  By the Koebe distortion theorem, $\confrad(z;U_z^{S_{z,w}}) \leq 4 r$, which gives the upper bound.
  By \cite[Lemma~\ref{I-lem::annulus-loop}]{extremes}, there is a loop contained in $B(z,r)$ but which
  surrounds $B(z,r/2^k)$ except with probability exponentially small $k$, which gives the lower bound.
\end{proof}

\begin{lemma}
\label{lem::first_split} There exists a constant $C>0$ (depending only on $\kappa$) such that
if $z,w \in D$ are distinct, and
$0<\eps<\min(|z-w|,\confrad(z;D))$,
then on the event $\{ \confrad(z;U_z^{S_{z,w}}) \geq 8\eps\}$,
\begin{multline}
\label{eq::log_bound}
 \bigg| \E\big[J^\cap_{z,\eps} - S_{z,w} \,|\, U_{z}^{S_{z,w}} \big] -  \E\big[J^\cap_{z,\eps} - S_{z,w} \big] -\\ \lptyp \log\frac{\confrad(z;U_z^{S_{z,w}})}{\min(|z-w|,\confrad(z;D))} \bigg|
\leq  C\,.
\end{multline}
\end{lemma}

\begin{proof}
Let $S=S_{z,w}$.
  By \eqref{eqn::contain_mean_cr_bound} of Corollary~\ref{cor::two_point_moments} we see that there exist
  $C_1 >0$ such that on the event $\{\confrad(z; U_z^S) \geq 8\eps\}$ we have
  \begin{equation}
  \label{eq::conditional}
  \left|\E\big[J_{z,\eps}^\cap - S_{z,w} \,|\, U_z^{S_{z,w}}\big]- \lptyp
      \log\frac{\confrad\big(z;U_z^{S_{z,w}}\big)}{\eps}\right|\leq C_1\,.
  \end{equation}
 We can write
  \begin{align}
     \E\left[J_{z,\eps}^\cap - S \right]
&=\E\left[(J_{z,\eps}^\cap - S)\one_{\{\confrad(z;U_z^S) \geq 8\eps\}} \right] +
     \E\left[(J_{z,\eps}^\cap - S)\one_{\{\confrad(z;U_z^S) < 8\eps\}} \right]\,. \label{eq::Jze-S}
  \end{align}

  Applying \eqref{eq::conditional}, we can write the first term of \eqref{eq::Jze-S} as,
\begin{align*}
  \E\left[(J_{z,\eps}^\cap - S)\one_{\{\confrad(z;U_z^S) \geq 8\eps\}} \right]
  &=
  \E\!\left[ \E[J_{z,\eps}^\cap - S\given U_z^S]\,
  \one_{\{\confrad(z;U_z^S) \geq 8\eps\}}\right] \\
  &= \E\!\left[\left(\lptyp\log \frac{\confrad(z;U_z^S)}{\eps} \pm C_1\right) \one_{\{\confrad(z;U_z^S) \geq 8\eps\}} \right] \\
  &= \lptyp \log\frac{\min(|z-w|,\confrad(z;D))}{\eps} \pm \text{const}\\
  &\quad - \E\!\left[\left(\lptyp\log \frac{\confrad(z;U_z^S)}{\eps}\right) \one_{\{\confrad(z;U_z^S) < 8\eps\}} \right]\,.
\end{align*}
Using \cite[Lemma~\ref{I-lem::annulus-loop}]{extremes}, there is a loop contained in $B(z,\eps)$ which surrounds $B(z,\eps/2^k)$ except with probability exponentially small in $k$, so the last term on the right is bounded by a constant (depending on $\kappa$).

If $J_{z,\eps}^\cap \geq S$, then $J_{z,\eps}^\cap - S$ counts the
number of loops $(\CL_z^k)_{k \in \N}$ after separating $z$
from $w$ before hitting $B(z,\eps)$.  If $J_{z,\eps}^\cap \leq
S$, then $S-J_{z,\eps}^\cap$ counts the number of loops
$(\CL_z^k)_{k \in \N}$ after intersecting $B(z,\eps)$ before
separating $z$ from $w$.  Consequently, by
Corollary~\ref{cor::loop_contain_stoch_dom}, we see that
absolute value of the second term of \eqref{eq::Jze-S}
is bounded by some constant $C_2 > 0$.
Putting these two terms of \eqref{eq::Jze-S} together, we obtain
\begin{equation} \label{eq::Jze-S-2}
  \left|\E\left[J_{z,\eps}^\cap - S_{z,w} \right] - \lptyp \log\frac{\min(|z-w|,\confrad(z;D))}{\eps} \right| \leq \text{const}\,.
\end{equation}

Subtracting \eqref{eq::Jze-S-2} from \eqref{eq::conditional} and rearranging gives~\eqref{eq::log_bound}.
\end{proof}

\begin{lemma} \label{lem::coalesce} Let $\{X_j\}_{j\in \N}$ be
  non-negative i.i.d.\ random variables whose law has a positive density
  with respect to Lebesgue measure on $(0,\infty)$ and for which there
  exists $\mgfparam_0 >0$ such that $\E[e^{\mgfparam_0 X_1}]<\infty$. For
  $a \geq 0 $, let $S^a_n=a+\sum_{j=1}^n X_j$, and for $a,M > 0$, let $\tau^a_M =
  \min\{n \geq 0 : S^a_n \geq M\}$.
  There exists a coupling between $S^a$ and
  $\widehat{S}^b$ (identically distributed to $S^b$ but not independent of it)
  and constants $C,c>0$ so that for all $0\leq a \leq b \leq M$, we have
  \[
  \P\bigg[S^a_{\tau^a_M} = \widehat{S}^b_{\widehat{\tau}^b_M}\bigg] \geq 1 - Ce^{-cM}.
  \]
\end{lemma}
Similar but non-quantitative convergence results are known for more general distributions
(for example, see \cite[Chapt.~3.10]{MR2489436}).
  For our results we need this convergence to be exponentially fast,
  for which we did not find a proof, so we provide one.
\begin{proof}[Proof of Lemma~\ref{lem::coalesce}]
  For $M > N > 0$, we construct a coupling between $\rho_N$ and $\rho_M$ as
  follows.  We take $S_0 = 0$ and $\widehat{S}_0=N-M$,
  and then take $\{X_j\}_{j \in \N}$ and $\{\widehat{X}_j\}_{j \in \N}$ to be
  two i.i.d.\ sequences with law as in the statement of
  the lemma, with the two sequences coupled with one another in a manner
  that we shall describe momentarily.
We let $S_n = \sum_{i=1}^n X_i$ and $\widehat{S}_n = \widehat{S}_0+\sum_{i=1}^n \widehat{X}_i$.
Define stopping times
\[ \tau_N = \min\{ n \geq 0 : S_n \geq N\} \quad\quad\text{and}\quad\quad
\widehat{\tau}_N = \min\{ n \geq 0 : \widehat{S}_n \geq N\}.\] Then
$S_{\tau_N} - N \sim \rho_N$ and $\widehat{S}_{\widehat{\tau}_N} - N \sim
\rho_M$.  We will couple the $X_j$'s and $\widehat{X}_j$'s so that with high
probability $S_{\tau_N} = \widehat{S}_{\widehat{\tau}_N}$.

Lemma~\ref{lem::firsthitting} implies that there exists a law $\wt{\rho}$ on
$(0,\infty)$ with exponential tails such that $\wt{\rho}$ stochastically
dominates $\rho_M$ for all $M > 0$.  We choose $\theta$ to be big enough
so that $\wt{\rho}([0,2\theta])\geq 1/2$.

We inductively define a sequence of
pairs of integers $(i_k,j_k)$ for $k\in\{0,1,2,\dots\}$ starting with
$(i_0,j_0)=(0,0)$.
If $S_{i_k}+\theta \leq \widehat{S}_{j_k}$ then we set
$(i_{k+1},j_{k+1})\colonequals(i_k+1,j_k)$ and sample $X_{i_{k+1}}$
independently of the previous random variables.  If
$\widehat{S}_{j_k}+\theta\leq S_{i_k}$, then we set
$(i_{k+1},j_{k+1})\colonequals(i_k,j_k+1)$ and sample
$\widehat{X}_{j_{k+1}}$ independently of the previous random variables.
Otherwise, $\big|S_{i_k}-\widehat{S}_{j_k}\big|\leq\theta$. In that case,
we set $(i_{k+1},j_{k+1})\colonequals(i_k+1,j_k+1)$ and sample
$(X_{i_{k+1}},\widehat{X}_{j_{k+1}})$ independently of the previous random
variables and coupled so as to maximize the probability that
$S_{i_{k+1}}=\widehat{S}_{j_{k+1}}$.  Note that once the walks coalesce,
they never separate.

We partition the set of steps into epochs.
We adopt the convention that the $k$th step is from time~$k-1$ to time~$k$.
The first epoch starts at time $k=0$.
For the epoch starting at time $k$ (whose first step is $k+1$), we let
\[\ell(k) = \min\left\{k'\geq k : \min(S_{i_{k'}},\wh{S}_{j_{k'}}) \geq \max(S_{i_{k}},\wh{S}_{j_{k}})-\theta \right\}.\]
Let $E_k$ be the event \[
E_k=\{|S_{i_{\ell(k)}}-\wh{S}_{j_{\ell(k)}}|\leq\theta\}. \] By our choice
of $\theta$, $\P[E_k] \geq 1/2$.  If event $E_k$ occurs, then we let
$\ell(k)+1$ be the last step of the epoch, and the next epoch starts at time
$\ell(k)+1$.  Otherwise,
we let $\ell(k)$ be the last step of the epoch, and the next epoch starts at time $\ell(k)$.

Let $D(t)$ denote the total variation distance between the law of $X_1$ and
the law of $t+X_1$.  Since $X_1$ has a density with respect to Lebesgue
measure which is positive in $(0,\infty)$, it follows that \[q\colonequals
\sup_{0\leq t \leq \theta} D(t) < 1.\] In particular, if the event $E$ occurs,
i.e., $\big|S_{i_{\ell(k)}}-\widehat{S}_{j_{\ell(k)}}\big|\leq\theta$, and the walks
have not already coalesced, then
$\P[S_{i_{\ell(k)+1}}\neq\widehat{S}_{j_{\ell(k)+1}}] \leq q$.

  Let $Y_k=\max(S_{i_k},\wh{S}_{j_k})$.  For the epoch starting at time $k$,
  the difference $Y_{\ell(k)}-Y_k$ is dominated by a random variable with
  exponential tails, since $\wt{\rho}$ has exponential tails. On the event
  $E_k$ there is one more step of size $Y_{\ell(k)+1}-Y_{\ell(k)}$ in the
  epoch. This step size is dominated by the maximum of two independent
  copies of the random variable $X_1$ and therefore has exponential tails.
  Thus if $k'$ is the start of the next epoch, then $Y_{k'}-Y_k$ is
  dominated by a fixed distribution (depending only on the law of $X_1$)
  which has exponential tails.  It follows from
  Cram\'er's theorem that for some $c >0$,
  it is exponentially unlikely that the number of epochs (before the
  walks overshoot $N$) is less than $c N$.

For each epoch, the walks have a $(1-q)\P[E_k]>0$
  chance of coalescing if they have not done so
  already.  After $c N$ epochs, the walkers have coalesced except
  with probability exponentially small in $N$, and except with exponentially
  small probability, these epochs all occur before the walkers overshoot $N$.
\end{proof}

\begin{lemma}
  \label{lem::first_split2}
There exist constants $C_3,c>0$ (depending only on $\kappa$) such that
if $z,w \in D$ are distinct, and
$0<\eps'\leq\eps\leq r$ where $r=\min(|z-w|,\confrad(z;D))$,
then
\begin{equation} \label{eq::exp_bound}
    \E\left[\left(\E\big[J^\cap_{z,\eps} - J^\cap_{z,\eps'} \,|\, U_{z}^{S_{z,w}} \big] -   \E[J^\cap_{z,\eps}
    - J^\cap_{z,\eps'}]\right)^2\right]  \leq C_3 \left(\frac{\eps}{r}\right)^{c}\,.
\end{equation}
\end{lemma}
\begin{proof}
  We construct a coupling between three $\CLE_\kappa$'s, $\Gamma$,
  $\wt{\Gamma}$, and $\acute\Gamma$, on the domain $D$.
  Let $S=S_{z,w}$, $\wt{S}=\wt{S}_{z,w}$,
  and $\acute S=\acute S_{z,w}$ denote the three corresponding stopping times.
  We take $\Gamma$ and $\acute\Gamma$ to be independent.
  On $D\setminus \acute{U}_z^{\acute S}$, we take $\wt{\Gamma}$ to be identical to $\acute\Gamma$.
  In particular, $\wt S=\acute S$ and $\wt{U}_z^{\wt S} = \acute{U}_z^{\acute S}$.  Within
  $\wt{U}_z^{\wt S}$, we couple $\wt\Gamma$ to $\Gamma$ as follows.
  We sample so that the sequences
  \[
  \left\{-\log \confrad\left(z; U_{z}^{S+k}\right) \right\}_{k\in
  \N}\quad \text{and}\quad \left\{-\log \confrad\left(z;
    \wt{U}_{z}^{\wt{S}+k}\right)\right\}_{k\in \N}
  \]
  are coupled as in Lemma~\ref{lem::coalesce}. Define
  \[
  K = \min\left\{k \geq S \,:\confrad\left(z;
      U_{z}^k\right)=\confrad\left(z;
      \wt{U}_{z}^{\wt{k}}\right)\text{ for some }\wt{k} \geq \wt{S}\right\}\,,
  \]
  and let $\wt{K}$ be the value of $\wt{k}$ for which the conformal radius
  equality is realized. Let $\psi:U_z^K\to \wt{U}_z^{\wt{K}}$ be the unique
  conformal map with $\psi(z)=z$ and $\psi'(z)>0$.  We take
  $\wt{\Gamma}$ restricted to
  $\wt{U}_z^{\wt{K}}$
  to be given by the image under $\psi$ of the restriction of
  $\Gamma$ to
  $U_z^K$.

  Since $|\log \confrad(z;U_z^S) - \log r|$ and $|\log \confrad(z;\wt{U}_z^{\wt{S}}) - \log r|$
  have exponential tails,
  and since the coupling time from Lemma~\ref{lem::coalesce} has exponential tails,
  each of $K-S$, $\wt{K}-\wt{S}$, and
  $|\log \confrad(z;U_z^K) - \log r| = |\log \confrad(z;\wt{U}_z^{\wt{K}}) - \log r|$
  have exponential tails, with parameters depending only on $\kappa$.

  Let
  \[
  \Delta\colonequals   \E[J_{z,\eps}^\cap - J_{z,\eps'}^\cap\,\given U_z^S]
  - \E[\wt{J}_{z,\eps}^\cap - \wt{J}_{z,\eps'}^\cap\,\given \wt{U}_z^{\wt{S}}]\,.
  \]
  In the above coupling $U_z^S$ and $\wt{U}_z^{\wt{S}}$ are independent, so we have
  \[
  \E[J_{z,\eps}^\cap - J_{z,\eps'}^\cap\,\given U_z^S] - \E[J_{z,\eps}^\cap
  - J_{z,\eps'}^\cap] =
  \E[\Delta \,\given U_z^S]\,.
  \]
  Therefore, the left-hand
  side of \eqref{eq::exp_bound} is equal to $\E[(\E[\Delta|U_z^S])^2]$. Jensen's inequality applied to the inner expectation yields
  \[
  \E[(\E[\Delta|U_z^S])^2] \leq \E[\E[\Delta^2\given U_z^S]] = \E[\Delta^2]\,.
  \]

  We can also write $\Delta$ as

  \begin{align*}
    \Delta&=\E\big[J_{z,\eps}^\cap - J_{z,\eps'}^\cap -
    \wt{J}_{z,\eps}^\cap +
    \wt{J}_{z,\eps'}^\cap\,\given U_z^S,\wt{U}_z^{\wt{S}}\big] \\
    &= \E\!\left[ J_{z,\eps}^\cap -K - \wt{J}_{z,\eps}^\cap +\wt{K}
      \given U_z^S,\wt{U}_z^{\wt{S}}\right]
      - \E\!\left[ J_{z,\eps'}^\cap - K - \wt{J}_{z,\eps'}^\cap +\wt{K}
          \given U_z^S,\wt{U}_z^{\wt{S}}\right]\,.
  \end{align*}

  and then use the inequality $(a+b)^2 \leq 2 (a^2+b^2)$ for $a,b\in\R$ to bound
  \[ \Delta^2 \leq 2 Y_\eps + 2 Y_{\eps'}\,, \]
where for $\hat\eps\leq\eps$ we define
\[
  Y_{\hat\eps} \colonequals \E\!\left[ J_{z,\hat\eps}^\cap - K -
      \wt{J}_{z,\hat\eps}^\cap + \wt{K} \,\given U_z^S,\wt{U}_z^{\wt{S}}\right]^2\,.
\]

  We define the event
  \[
    A = \{ \confrad(z;U_z^K) \geq \sqrt{r \eps}\}\,.
  \]
Then
\begin{align*}
\E[Y_{\hat\eps}\,\one_A]
&=  \E\!\left[\E\!\left[ J_{z,\hat{\eps}}^\cap -K - \wt{J}_{z,\hat{\eps}}^\cap +\wt{K} \,\given U_z^S,\wt{U}_z^{\wt{S}}\right]^2 \one_A \right]\\
 &\leq \E\!\left[\E\!\left[ (J_{z,\hat{\eps}}^\cap - K - \wt{J}_{z,\hat{\eps}}^\cap +\wt{K})^2\,\one_A  \,\big|\, U_z^S,\wt{U}_z^{\wt{S}} \right]\right]\\
 &= \E\!\left[ (J_{z,\hat{\eps}}^\cap - K - \wt{J}_{z,\hat{\eps}}^\cap +\wt{K})^2\,\one_A\right]\\
 &\leq \const\times(\eps/r)^c
  \end{align*}
  where the last inequality follows from Lemma~\ref{lem::mean_loops_lcr},
 for some $c>0$ and for suitably large $r/\eps$.

  Next we apply Cauchy-Schwarz to find that
  \[
  \E[Y_{\hat\eps}\one_{A^c}] \leq \sqrt{\E[Y_{\hat\eps}^2]\P[A^c]}.
   \]
  Lemma~\ref{lem::coalesce} and the construction of the coupling between
  $\Gamma$ and $\wt{\Gamma}$ imply that $\P[A^c] \leq \text{const} \times
  (\eps/r)^c$ for some $c>0$. It therefore suffices to show that
  $\E[Y_{\hat\eps}^2]\leq C$ for some constant $C$ which does not depend on
  $\eps$ or $\eps'$.  By Jensen's inequality, it suffices to show that
  there exists $C$ such that
  \begin{equation} \label{eq::bounded_fourth}
  \E[(J_{z,\hat\eps}^\cap - K - \wt{J}_{z,\hat\eps}^\cap + \wt{K})^4]
    \leq C\,.   \end{equation}

  To prove \eqref{eq::bounded_fourth}, we consider the event $B=\{
  \confrad(z;U_z^K)\geq \eps\}$.
  By Lemma~\ref{lem::mean_loops_lcr},
  \[\E[(J_{z,\hat\eps}^\cap - K - \wt{J}_{z,\hat\eps}^\cap + \wt{K})^4 \one_{B}] \leq \const\]
  where the constant depends only on $\kappa$.

  Using $(a+b)^4\leq 8(a^4+b^4)$ for $a,b\in\R$, and the fact that
  $J_{z,\hat\eps}^\cap - K$ and $\wt{J}_{z,\hat\eps}^\cap - \wt{K}$
  are equidistributed, we have
  \[ \E[(J_{z,\hat\eps}^\cap - K - \wt{J}_{z,\hat\eps}^\cap + \wt{K})^4 \,\one_{B^c}] \leq 16\, \E[(J_{z,\hat\eps}^\cap - K)^4\,\one_{B^c}]\,.\]
On the event $B^c$, we have $K\geq J_{z,\hat\eps}^\cap$.  Conditional on this, $K-J_{z,\hat\eps}^\cap$
has exponentially decaying tails, so the above fourth moment is bounded by a constant (depending on $\kappa$),
which completes the proof.
\end{proof}

\begin{lemma} \label{lem::cov_bound}
  Suppose $0<\eps_1(z)\leq\eps$ and $0<\eps_2(z)\leq\eps$ on a compact subset $K\subset D$ of the domain $D$.
  Then there is some $c>0$ (depending on $\kappa$) and $C_0>0$ (depending on $\kappa$, $D$, and $K$) for which
    \begin{equation}
    \label{eq::cov_bound}
    \iint\limits_{K \times K} |\cov(\Loopcount_z(\eps_1(z)) - \Loopcount_z(\eps_2(z)),
    \Loopcount_w(\eps_1(w)) -
    \Loopcount_w(\eps_2(w)))| \, dz\,dw \leq C_0 \eps^c\,.
  \end{equation}
\end{lemma}

\begin{proof}
  For a random variable $X$, we let $\rc{X}$ denote
  \begin{equation}
    \label{eqn::loopcount_bar_def}
    \rc{X} = X - \E[X]\,.
  \end{equation}
  We let $Y_z$ denote
  \begin{equation}
   Y_{z}\colonequals {J}_{z,\,\eps_1(z)}^\cap-{J}_{z,\,\eps_2(z)}^\cap\,.
  \end{equation}
  Recalling that $J_{z,r}^{\cap}=\Loopcount_z(r)+1$, we see that
  \[ \E[\rc{Y}_z \rc{Y}_w] = \cov(\Loopcount_z(\eps_1(z)) - \Loopcount_z(\eps_2(z)),
    \Loopcount_w(\eps_1(w)) - \Loopcount_w(\eps_2(w))) \,,\]
   so we need to bound $\big|\E[\rc{Y}_z \rc{Y}_w]\big|$.

   We treat two subsets of $K\times K$ separately: (1) the near regime
   $\{(z,w) \,:|z-w|\leq\eps\}$, and (2) the far regime
   $\{(z,w)\,:\eps<|z-w|\}$.

For the near regime, we first write \[Y_z = Y_{z,w}^{(1)} +
Y_{z,w}^{(2)}\,,\] where $Y_{z,w}^{(1)}$ counts those loops surrounding
$B(z,\min(\eps_1(z),\eps_2(z)))$ and intersecting
$B(z,\max(\eps_1(z),\eps_2(z)))$ with index smaller than $S_{z,w}$, and
$Y_{z,w}^{(2)}$ counts those loops with index at least $S_{z,w}$. Then
$\Sigma_{z,w}$ determines $Y_{z,w}^{(1)}$ and $Y_{w,z}^{(1)}$, and
conditional on $\Sigma_{z,w}$, $Y_{z,w}^{(2)}$ and $Y_{w,z}^{(2)}$ are
independent (recall that $\Sigma_{z,w}$ was defined in \eqref{eq::sigma}).
Thus $Y_{z,w}^{(i)}$ and $Y_{w,z}^{(j)}$ are conditionally independent
(given $\Sigma_{z,w}$) for $i,j\in\{1,2\}$.

   Observe that
   \begin{equation}\label{eq::sum-of-4}
     \big|\E[\rc{Y}_z\rc{Y}_w]\big| \leq \sum_{i,j\in\{1,2\}} \big|\E[\rc{Y}_{z,w}^{(i)}\rc{Y}_{w,z}^{(j)}]\big|\,.
   \end{equation}
   For $i,j\in\{1,2\}$,
   \begin{align}
     \left|\E\big[\rc{Y}_{z,w}^{(i)}\rc{Y}_{w,z}^{(j)}\big]\right|
     &= \left|\E\left[\E\big[\rc{Y}_{z,w}^{(i)}\rc{Y}_{w,z}^{(j)}\given\Sigma_{z,w}\big]\right]\right| \notag\\
     &= \left|\E\left[\E\big[\rc{Y}_{z,w}^{(i)}\given\Sigma_{z,w}\big] \, \E\big[\rc{Y}_{w,z}^{(j)}\given\Sigma_{z,w}\big]\right]\right| \notag\\
     &\leq \E\left[\E\big[\rc{Y}_{z,w}^{(i)}\given\Sigma_{z,w}\big]^2\right]^{1/2} \, \E\left[\E\big[\rc{Y}_{w,z}^{(j)}\given\Sigma_{z,w}\big]^2\right]^{1/2}\,. \label{eq::conditional-CS}    \end{align}

For the index $i=1$, we write
\begin{align*}
  \E\left[\E\big[\rc{Y}_{z,w}^{(1)}\given\Sigma_{z,w}\big]^2\right] = \E\left[\big(\rc{Y}_{z,w}^{(1)}\big)^2\right] \leq \E\left[(Y_{z,w}^{(1)})^2\right]
  &= \E\!\left[\E\big[(Y_{z,w}^{(1)})^2\big]\,\big|\, U_z^{J_{z,\eps}^\cap} \right]\,.
\end{align*}
But
\[
  Y_{z,w}^{(1)} \leq 1+\Loopcount_{z,w}\left(\Gamma|_{U_z^{J_{z,\eps}^\cap}}\right)\,.
\]
By Theorem~\ref{thm::surround_two_point_bound},
$\E[(1+\Loopcount_{z,w}(\Gamma|_U))^2] \leq\const+\const\times G_{U}(z,w)^2$, where
$G_U$ denotes the Green's function for the Laplacian in the domain $U$.
By the Koebe distortion theorem, the Green's function is in turn bounded by
$G_U(z,w) \leq \const+\const\times\max(0,\log(\confrad(z;U)/|z-w|))$.
Therefore,
  \[
  \E\left[\big(\rc{Y}_{z,w}^{(1)}\big)^2\right]
   \leq \E\left[O\left(1+\log^2\frac{|z-w|}{\confrad\big(z;U_z^{J_{z,\eps}^\cap}\big)}\right)\right]\,.
  \]
  By Lemma~\ref{lem::firsthitting}, $-\log \confrad\big(z;U_z^{J_{z,\eps}^\cap}\big) = -\log \eps + X$
  for some random variable $X$ with exponentially decaying tails. It
  follows that
  \begin{equation}
  \E\left[\E\big[\rc{Y}_{z,w}^{(1)}\given\Sigma_{z,w}\big]^2\right] =
   \E\left[\big(\rc{Y}_{z,w}^{(1)}\big)^2\right] = O\left(1+\log^2\frac{|z-w|}{\eps}\right)\,. \label{eq::Y1-bound}
  \end{equation}

  For the index $i=2$, we express $Y_{z,w}^{(2)}$ in terms of $J_{z,\eps_1(z)}$ and $J_{z,\eps_2(z)}$ and use Lemma~\ref{lem::first_split} twice (once with
  $\eps_1(z)$ and once with $\eps_2(z)$ playing the role of $\eps$ in the
  lemma statement) and subtract to write
\begin{align}
\E\big[\rc{Y}_{z,w}^{(2)}\given \Sigma_{z,w}\big] =
\E\big[\rc{Y}_{z,w}^{(2)}\given U_z^{S_{z,w}}\big]  &\leq \const \notag\\
 \E\left[\E\big[\rc{Y}_{z,w}^{(2)}\given
    \Sigma_{z,w}\big]^2\right] &\leq C \,.\label{eq::Y2-bound}
\end{align}
for some constant $C$ depending only on $\kappa$.

Combining \eqref{eq::sum-of-4}, \eqref{eq::conditional-CS}, \eqref{eq::Y1-bound}, and \eqref{eq::Y2-bound},
we obtain
\begin{align*}
\big|\E[\rc{Y}_z \rc{Y}_w]\big| &\leq
\const+\const\times\log^2\frac{\eps}{|z-w|}\notag,
\end{align*}
which implies
\begin{align}
\iint\limits_{\substack{K\times K\\ |z-w|\leq \eps}} \big|\E[\rc{Y}_z
\rc{Y}_w]\big|\,dz\,dw &\leq
\const\times\area(K)\times \eps^2\,. \label{eq::intermediate-regime}
\end{align}

For the far regime, we again condition on $\Sigma_{z,w}$, the loops up to and including the first
ones separating $z$ from $w$, and use Cauchy-Schwarz, as in \eqref{eq::conditional-CS},
but without first expressing $Y_z$ and $Y_w$ as sums:
\begin{equation}
 \left|\E\big[\rc{Y}_z\rc{Y}_w\big]\right|
  \leq \E\left[\E\big[\rc{Y}_z\given\Sigma_{z,w}\big]^2\right]^{1/2} \, \E\left[\E\big[\rc{Y}_w\given\Sigma_{z,w}\big]^2\right]^{1/2}\,.
  \label{eq::conditional-CS-2}
\end{equation}
  By Lemma~\ref{lem::first_split2}, we have
  \begin{equation}
    \label{eq::just_z}
    \E\big[
      \E[\rc{Y}_z\given\Sigma_{z,w}]^2\big] \leq C \left(\frac{\eps}{\min(|z-w|,\confrad(z;D))}\right)^c\,.
  \end{equation}
  Integrating over $\{(z,w)\in K\times K\,:\,\eps<|z-w|\}$ gives
  \eqref{eq::cov_bound}.
\end{proof}

\section{Properties of Sobolev spaces}
\label{sec::sobolev_spaces}

\makeatletter{}In this section we provide an overview of the distribution theory and
Sobolev space theory required for the proof of
Theorem~\ref{thm::existence_weighted}. We refer the reader to
\cite{tao-epsilon} or \cite{TAYLOR_PDE} for a more detailed introduction.

Fix a positive integer $d$. Recall that the Schwartz space $\Sc(\R^d)$ is
defined to be the set of smooth, complex-valued functions on $\R^d$
whose derivatives of all orders decay faster than any polynomial at
infinity. If $\beta = (\beta_1,\beta_2,\ldots,\beta_d)$ is a multi-index,
then the partial differentiation operator $\partial^\beta$ is defined by
$\partial^\beta=\partial_{x_1}^{\beta_1} \partial_{x_2}^{\beta_2}\cdots \partial_{x_d}^{\beta_d}$.
We equip $\Sc(\R^d)$ with the topology generated by the family of seminorms
\[
\left\{\raisebox{2pt}{$\displaystyle\|\phi\|_{n,\beta} \colonequals \sup_{x\in
      \R^d} |x|^n |\partial^\beta \phi(x)| \,:\,n \geq 0, \: \beta\text{ is a
      multi-index}$} \right\}\,.
\]
The space $\Sc'(\R^d)$ of tempered distributions is defined to be the space
of continuous linear functionals on $\Sc(\R^d)$.  We write the evaluation
of $f\in \Sc'(\R^d)$ on $\phi \in \Sc(\R^d)$ using the notation
$\ang{f,\phi}$.  For any Schwartz function $g\in\Sc(\R^d)$ there is an
associated continuous linear functional $\phi\mapsto \int_{\R^d}
g(x)\phi(x)\,dx$ in $\Sc'(\R^d)$, and $\Sc(\R^d)$ is a dense subset of
$\Sc'(\R^d)$ with respect to the weak* topology.

For $\phi\in \Sc(\R^d)$, its Fourier transform $\wh{\phi}$ is defined by
\[
\wh{\phi}(\xi) = \int_{\R^d} e^{-2\pi i x\cdot \xi}\phi(x)\,dx \quad \text{for
}\xi\in \R^d\,.
\]
Since $\phi\in \Sc(\R^d)$ implies $\wh{\phi}\in \Sc(\R^d)$
\cite[Section~1.13]{tao-epsilon} and since
$\ang{\wh{\phi}_1,\phi_2}=\iint\phi_1(x) e^{-2\pi i x\cdot y}
\phi_2(y)\,dx\,dy =\ang{\phi_1,\wh{\phi}_2}$ for all
$\phi_1,\phi_2\in\Sc(\R^d)$, we may define the Fourier transform $\wh{f}$
of a tempered distribution $f\in\Sc'(\R^d)$ by setting
$\ang{\wh{f},\phi}\colonequals \ang{f,\wh\phi}$ for each
$\phi\in\Sc(\R^d)$.

For $x\in \R^d$, we define $\ang{x}\colonequals (1+|x|^2)^{1/2}$. For
 $s\in \R$, define $H^s(\R^d) \subset \Sc'(\R^d)$ to be the set of
  functionals $f$ for which there exists $R^s_f\in L^2(\R^d)$ such that for
  all $\phi\in \Sc(\R^d)$,
\begin{equation} \label{eq::defHs}
\ang{\wh{f},\phi} = \int_{\R^d} R^s_f(\xi) \phi(\xi) \ang{\xi}^{-s}\,d\xi\,.
\end{equation}
Equipped with the inner product
\begin{equation} \label{eq::ft}
\ang{f,g}_{H^s(\R^d)} \colonequals \int_{\R^d} R^s_f(\xi) \overline{R^s_g(\xi)} \,d\xi\,,
\end{equation}
$H^s(\R^d)$ is a Hilbert space.  (The space $H^s(\R^d)$ is the same
as the Sobolev space denoted $W^{s,2}(\R^d)$ in the literature.)

Recall that the support of a function $f:\R^d \to \C$ is defined to the
closure of the set of points where $f$ is nonzero. Define $\T=[-\pi,\pi]$
with endpoints identified, so that $\T^d$, the $d$-dimensional torus, is a
compact manifold. If $M$ is a manifold (such as $\R^d$ or $\T^d$), we
denote by $C_c^\infty(M)$ the space of smooth, compactly supported
functions on $M$. We define the topology of $C_c^\infty(M)$ so that $\psi_n
\to \psi$ if and only if there exists a compact set $K\subset M$ on which
each $\psi_n$ is supported and $\partial^\alpha \psi_n \to \partial^\alpha
\psi$ uniformly, for all multi-indices $\alpha$ \cite{tao-epsilon}. We
write $C_c^\infty(M)'$ for the space of continuous linear functionals on
$C_c^\infty(M)$, and we call elements of $C_c^\infty(M)'$
\textit{distributions\/} on $M$. For $f\in C_c^\infty(\T^d)'$ and $k\in
\Z^d$, we define the Fourier coefficient $\wh{f}(k)$ by evaluating $f$ on
the element $x\mapsto e^{-ik\cdot x}$ of $C_c^\infty(\T^d)$. For
distributions $f$ and $g$ on $\T^d$, we define an inner product with
Fourier coefficients $\wh{f}(k)$ and $\wh{g}(k)$:
\begin{equation} \label{eq::fc}
\ang{f,g}_{H^s(\T^d)} \colonequals \sum_{k\in \Z^d} \ang{k}^{2s}\wh{f}(k) \overline{\wh{g}(k)}\,.
\end{equation}
If $f\in \Sc'(\R^d)$ is supported in $(-\pi,\pi)^d$, i.e.\ vanishes on functions which are supported in the complement of $(-\pi,\pi)^d$, then $f$ can
be thought of as a distribution on $\T^d$, and the norms corresponding to
the inner products in \eqref{eq::ft} and \eqref{eq::fc} are equivalent
\cite{TAYLOR_PDE} for such distributions $f$.

Note that $H^{-s}(\R^d)$ can be identified with the dual of $H^s(\R^d)$: we
associate with $f\in H^{-s}(\R^d)$ the functional $g\mapsto \ang{f,g}$
defined for $g\in H^{s}(\R^d)$ by
\[
\ang{f,g} \colonequals \int_{\R^d} R^{-s}_f(\xi) \overline{R^s_g(\xi)}\,d\xi\,.
\]
This notation is justified by the fact that when $f$ and $g$ are in
  $L^2(\R^d)$, this is the same as the $L^2(\R^d)$ inner product of $f$ and
  $g$.  By Cauchy-Schwarz, $g\mapsto \ang{f,g}$ is a bounded linear
functional on $H^s(\R^d)$. Observe that the operator topology on the
  dual $H^s(\R^d)$ coincides with the norm topology of $H^{-s}(\R^d)$ under
  this identification.

It will be convenient to work with local versions of the Sobolev spaces
$H^s(\R^d)$. If $h\in \Sc'(\R^d)$ and $\psi\in C_c^\infty(\R^d)$,
we define the product $\psi h \in \Sc'(\R^d)$ by $\ang{\psi h,f} =
\ang{h,\psi f}$. Furthermore, if $h\in H^{s}(\R^d)$, then $\psi h\in H^s(\R^d)$
as well \cite[Lemma~4.3.16]{ban-crainic}.  For $h\in C_c^\infty(D)'$, we
say that $h \in \Hloc^{s}(D)$ if $\psi h \in H^{s}(\R^d)$ for
every $\psi \in C_c^\infty(D)$. We equip $\Hloc^{s}(D)$ with a
topology generated by the seminorms $\|\psi\cdot\|_{H^{s}(\R^d)}$, which
implies that $h_n \to h$ in $\Hloc^{s}(D)$ if and only if $\psi
h_n \to \psi h$ in $H^{s}(\R^d)$ for all $\psi \in C_c^\infty(D)$.

The following proposition provides sufficient conditions for proving
  almost sure convergence in $\Hloc^{-d-\delta}(\R^d)$.

  \begin{proposition} \label{prop::convergence_Hminusd} Let $D\subset \R^d$
    be an open set, let $\delta>0$, and suppose that $(f_n)_{n\in \N}$ is a
    sequence of random measurable functions defined on $D$.  Suppose further
    that for every compact set $K\subset D$,
\[
\int_K \E\big[|f_n(x)|^2\big]\,dx < \infty
\]
 and there exist a summable
    sequence $(\bound_n)_{n\in \N}$ of positive real numbers such that for
    all $n\in \N$, we have
\begin{equation} \label{eq::cov_decay}
\iint_{K\times K} \big| \E[(f_{n+1}(x)-f_{n}(x)) (f_{n+1}(y)-f_n(y))] \big|
\,dx\,dy \leq \bound_n^3\,.
\end{equation}
Then there exists $f\in \Hloc^{-d-\delta}(\R^d)$ supported on the closure
of $D$ such that $f_n \to f$ in $\Hloc^{-d-\delta}(D)$ almost surely.
\end{proposition}

Before proving Proposition~\ref{prop::convergence_Hminusd}, we prove the
following lemma. Recall that a sequence $(K_n)_{n\in \N}$ of compact sets
is called a \textit{compact exhaustion\/} of $D$ if $K_n\subset K_{n+1}
\subset D$ for all $n\in \N$ and $D=\bigcup_{n\in \N} K_n$.

\begin{lemma} \label{lem::Hsloc} Let $s>0$, let $D\subset \R^d$ be an open
  set, suppose that $(K_j)_{j\in \N}$ is a compact exhaustion of $D$,
  and let $(f_n)_{n \in \N}$ be a sequence of elements of
  $H^{-s}(\R^d)$. Suppose further that $(\psi_j)_{j\in \N}$ satisfies
  $\psi_j \in C_c^\infty(D)$ and $\left.\psi_j\right|_{K_j}=1$ for all
  $j\in \N$. If for every $j$ there exists $f^{\psi_j} \in H^{-s}(\R^d)$
  such that $\psi_j f_n \to f^{\psi_j}$ as $n\to\infty$ in $H^{-s}(\R^d)$,
  then there exists $f\in \Hloc^{-s}(D)$ such that $f_n \to f$ in
  $\Hloc^{-s}(D)$.
\end{lemma}

\begin{proof}
  We claim that for all $\psi \in C_c^\infty(D)$, the sequence $\psi f_n$
  is Cauchy in $H^{-s}(\R^d)$. We choose $j$ large enough that $\supp\psi\subset K_j$.
  For all $g\in H^s(\R^d)$,
  \[
  |\ang{\psi f_n,g} - \ang{\psi f_m,g}| = |\ang{\psi_j(f_n-f_m),\psi g}|\,.
  \]
  By hypothesis $\psi_j f_n$ converges in $H^{-s}(\R^d)$ as $n\to\infty$,
  so we may take the supremum over $\{g\,:\,\|g\|_{H^s(\R^d)}\leq 1\}$ of
  both sides to conclude $\|\psi f_n-\psi f_m\|_{H^{-s}(\R^d)}\to 0$ as
  $\min(m,n)\to\infty$. Since $H^{-s}(\R^d)$ is complete, it follows that
  for every $\psi \in C_c^\infty(D)$, there exists $f^\psi \in
  H^{-s}(\R^d)$ such that $\psi f_n \to f^\psi$ in $H^{-s}(\R^d)$.

  We define a linear functional $f$ on $C_c^\infty(D)$ as follows.
  For $g\in C_c^\infty(D)$, set
  \begin{equation} \label{eq::hdef}
  \ang{f,g} \colonequals \ang{f^{\psi},g}\,,
  \end{equation}
  where $\psi$ is a smooth compactly supported function which is
  identically equal to 1 on the support of $g$. To see that this definition
  does not depend on the choice of $\psi$, suppose that $\psi_1\in
  C_c^\infty(D)$ and $\psi_2\in C_c^\infty(D)$ are both equal to 1 on the
  support of $g$. Then we have
  \[
  \ang{f^{\psi_1},g} - \ang{f^{\psi_2},g} = \lim_{n\to\infty}\ang{(\psi_1-\psi_2)f_n,g}
  = 0\,,
  \]
  as desired. From the definition in \eqref{eq::hdef}, $f$ inherits
  linearity from $f^\psi$ and thus defines a linear functional on
  $C_c^\infty(D)$.  Furthermore, $f\in \Hloc^{-s}(D)$ since $\psi f =
  f^\psi \in H^{-s}(R^d)$ for all $\psi\in C_c^\infty(D)$. Finally, $f_n \to
  f$ in $\Hloc^{-s}(D)$ since $\psi f_n \to \psi f=f^\psi$ in
  $H^{-s}(\R^d)$.
\end{proof}

\begin{proof}[Proof of Proposition~\ref{prop::convergence_Hminusd}]
  Fix $\psi\in C_c^\infty(D)$. Let $D_\psi$ be a bounded open set
  containing the support $\psi$ and whose closure is contained in
  $D$. Since $D_\psi$ is bounded, we may scale and translate it so that it
  is contained in $(-\pi,\pi)^d$.  We will calculate the Fourier
  coefficients of $\psi (f_{n+1}-f_n)$ in $(-\pi,\pi)^d$, identifying it with
  $\T^d$. By Fubini's theorem, we have for all $k\in \Z^d$
\begin{align}
\label{eqn::diff_bound}
\E|&\widehat{\psi f_{n+1}-\psi f_n}(k)|^2  \\ \nonumber &= \E \left[\left( \int_D
    f_{n+1}(x)\psi(x) e^{-ik\cdot x} dx - \int_D f_{n}(x)\psi(x)e^{-ik\cdot x} dx
  \right)^2\right] \\ \nonumber &\leq \|\psi \|^2_{L^\infty(\R^d)}
\iint\limits_{D_\psi\times D_\psi} \big|\E[ (f_{n+1}(x) - f_n(x))(f_{n+1}(y) -
f_n(y))]\big| \, dx \,dy \\ \nonumber
&\leq \,\|\psi\|^2_{L^\infty(\R^d)}\,\bound_n^3\,,
\end{align}
by \eqref{eq::cov_decay}.  By Markov's inequality,
\eqref{eqn::diff_bound} implies
\[
\P\left[|\widehat{\psi f_{n+1}-\psi f_{n}}(k)| \geq \bound_n
  \ang{k}^{d/2+\delta/2}\right]
\leq \|\psi\|^2_{L^\infty(\R^d)}\,\ang{k}^{-d-\delta}\bound_n\,.
\]
The right-hand side is summable in $k$ and $n$, so by the Borel-Cantelli
lemma, the event on the left-hand side occurs for at most finitely many
pairs $(n,k)$, almost surely.  Therefore, for sufficiently large $n_0$,
this event does not occur for any $n\geq n_0$. For these values of $n$, we
have
\begin{align*}
\|\psi f_n - \psi f_{n+1}\|^2_{H^{-d-\delta}(\T^d)}
&= \sum_{k\in \Z^d} |\widehat{\psi f_n - \psi f_{n+1}}(k)|^2 \ang{k}^{-2(d+\delta)} \\
&\leq \sum_{k\in \Z^d} \bound_n^2\ang{k}^{d+\delta}\ang{k}^{-2d-2\delta} \\
&= O(\bound_n^2/\delta)\,,
\end{align*}
Applying the triangle inequality, we find that for $m,n\geq n_0$
\begin{equation} \label{eq::cauchy}
\| \psi f_m - \psi f_n \|_{H^{-d-\delta}(\T^d)} =O\left(\delta^{-1/2}\sum_{j=m}^{n-1}\bound_j\right)\,.
\end{equation}
Recall that the $H^{-d-\delta}(\T^d)$ and $H^{-d-\delta}(\R^d)$ norms are
equivalent for functions supported in $(-\pi,\pi)^d$ (see the discussion
following \eqref{eq::fc}).  The sequence $(\bound_n)_{n\in \N}$ is summable
by hypothesis, so \eqref{eq::cauchy} shows that $(\psi f_n)_{n\in \N}$ is
almost surely Cauchy in $H^{-d-\delta}(\R^d)$.  Since $H^{-d-\delta}(\R^d)$
is complete, this implies that with probability 1 there exists $h^\psi \in
H^{-d-\delta}(\R^d)$ to which $\psi f_n$ converges in the operator topology
on $H^{-d-\delta}(\R^d)$.

By assumption $f_n\in \Hloc^0(\R^d)$, so $f_n\in H^{-d-\delta}(\R^d)$.  We may then
apply Lemma~\ref{lem::Hsloc} to obtain a limiting random variable $f\in
\Hloc^{-d-\delta}(\R^d)$ to which $(f_n)_{n\in \N}$ converges in
$\Hloc^{-d-\delta}(\R^d)$.
\end{proof}

\section{\texorpdfstring{Convergence to limiting field}{Convergence to limiting field}}
\label{sec::weighted_loop_distribution}

\makeatletter{}

We have most of the ingredients in place to prove the convergence of the
centered $\eps$-nesting fields, but we need one more lemma.

\begin{lemma} \label{lem::monotone_plus}
  Fix $C>0$, $\alpha>0$, and $L\in \R$. Suppose that $F, F_1,$ and $F_2$ are
  real-valued functions on $(0,\infty)$ such that
  \begin{enumerate}[(i)]
  \item $F_1$ is nondecreasing on $(0,\infty)$,
  \item $|F_2(x+\delta)-F_2(x)|\leq C \max(\delta^\alpha,e^{-\alpha x})$ for
    all $x>0$ and $\delta>0$,
  \item $F=F_1+F_2$, and
  \item For all $\delta>0$, $F(n\delta) \to L$ as $n\to\infty$ through the
    positive integers.
  \end{enumerate}
  Then $F(x) \to L$ as $x\to\infty$.
\end{lemma}

\begin{proof}
  Let $\eps>0$, and choose $\delta>0$ so that $C\delta^\alpha < \eps$. Choose
  $x_0$ large enough that $C e^{-\alpha x_0}<\eps$ and $|F(n\delta)-L|< \eps$
  for all $n>x_0/\delta$. Fix $x>x_0$, and define $a=\delta \lfloor
  x/\delta \rfloor$. For $u\in \{F,F_1,F_2\}$, we write $\Delta
  u=u(a+\delta)-u(a)$. Observe that $|\Delta F_2|\leq \eps$ by (ii). By
  (iii) and (iv), this implies
  \[
  |\Delta F_1| = |\Delta F - \Delta F_2| \leq |\Delta F| + |\Delta F_2| < 3\eps\,.
  \] Since $F_1$ is monotone, we get $|F_1(x)-F_1(a)|< 3\eps$. Furthermore,
  (ii) implies $|F_2(x)-F_2(a)|< \eps$. It follows that
  \[
  |F(x)-L| \leq |F_1(x)-F_1(a)|+|F_2(x)-F_2(a)|+|F(a)-L| < 5\eps\,.
  \]
  Since $x>x_0$ and $\eps>0$ were arbitrary, this concludes the proof.
\end{proof}

\begin{theorem} \label{thm::almost_sure_norm} Let $h_\eps(z)$ be the
  centered weighted nesting of a $\CLE_\kappa$ around the ball $B(z,\eps)$,
  defined in \eqref{eq::h_eps}.  Suppose $0<a<1$.  Then $(h_{a^n})_{n\in
    \N}$ almost surely converges in $\Hloc^{-2-\delta}(D)$.
\end{theorem}
\begin{proof}
Immediate from Theorem~\ref{thm::var-field-diff} and Proposition~\ref{prop::convergence_Hminusd}.
\end{proof}

\begin{proof}[Proof of Theorem~\ref{thm::existence_weighted}]
  We claim that for all $g\in C_c^\infty(D)$, we have $\ang{h_\eps,g}\to
  \ang{h,g}$ almost surely. Suppose first that the loop weights are almost
  surely nonnegative and that $g\in C_c^\infty(D)$ is a nonnegative test
  function. Define $F(x)\colonequals \ang{h_{e^{-x}},g}$, $F_1(x)\colonequals
  \ang{S_z(e^{-x}),g}$, and $F_2(x)\colonequals-\ang{\E[S_z(e^{-x})],g}$.
  We apply Lemma~\ref{lem::monotone_plus} with $\alpha$ as given in
  Lemma~\ref{lem::mean_loops_inr}, which implies
\begin{equation} \label{eq::limit_positive}
\lim_{\eps\to 0} \ang{h_\eps,g} = \ang{h,g} \quad \text{for} \quad g\in
C_c^\infty(D), g\geq 0\,.
\end{equation}

For arbitrary $g\in C_c^\infty(D)$, we choose $\tilde{g}\in C_c^\infty(D)$
so that $\tilde{g}$ and $g+\tilde{g}$ are both nonnegative. Applying
\eqref{eq::limit_positive} to $\tilde{g}$ and $g+\tilde{g}$, we see that
  \begin{equation} \label{eq::limit_compact}
    \lim_{\eps\to 0} \ang{h_\eps,g} = \ang{h,g} \quad \text{for} \quad g\in
    C_c^\infty(D)\,.
  \end{equation}
  Finally, consider loop weights which are not necessarily
  nonnegative. Define loop weights
  $\xi^{\pm}_\Loop=(\xi_\Loop)^{\pm}$, where $x^+=\max(0,x)$ and
  $x^-=\max(0,-x)$ denote the positive and negative parts of $x\in
  \R$. Define $h^{\pm}$ to be the weighted nesting fields associated with
  the weights $\xi^{\pm}_\Loop$ (associated with the same $\CLE$). Then $\ang{h^{\pm}_{\eps},g}\to \ang{h^{\pm},g}$
  almost surely, and
  \[
  \ang{h_\eps,g} = \ang{h_\eps^+,g}-\ang{h_\eps^-,g} \to \ang{h^+,g}-\ang{h^-,g} = \ang{h,g}\,,
  \]
  which concludes the proof that $\ang{h_\eps,g} \to \ang{h,g}$ almost surely.

  To see that the field $h$ is measurable with respect to the
  $\sigma$-algebra $\Sigma$ generated by the $\CLE_\kappa$ and the weights
  $(\xi_\CL)_{\Loop\in\Gamma}$, note that there exists a countable
    dense subset $\mathcal{F}$ of $C_c^\infty(D)$
    \cite[Exercise~1.13.6]{tao-epsilon}. Observe that $h_{2^{-n}}$ is
  $\Sigma$-measurable and $h$ is determined by the values
  $\{h_{2^{-n}}(g)\,:n\in \N,g\in \mathcal{F}\}$. Since $h$ is an almost
  sure limit of $h_{2^{-n}}$, we conclude that $h$ is also $\Sigma$-measurable.

  To establish conformal invariance, let $z\in D$ and $\eps>0$ and define
  the sets of loops
\begin{align*}
  \Xi_1 &= \text{loops surrounding }B(\varphi(z),\eps |\varphi'(z)|),\text{ and} \\
  \Xi_2 &= \text{loops surrounding }\varphi(B(z,\eps)) \\
  \Xi_3 &= \Xi_1 \Delta\,\Xi_2\,,
\end{align*}
where $\Delta$ denotes the symmetric difference of two sets.
Since either $\Xi_1\subset\Xi_2$ or $\Xi_2\subset\Xi_1$,
\begin{equation*}
h_\eps(z) - \acute{h}_{\eps|\varphi'(z)|}(\varphi(z))
 = \pm\sum_{\xi \in \Xi_3} \xi_\CL\,.
\end{equation*}
Multiplying by $g\in C_c^\infty(D)$, integrating over $D$, and taking $\eps
\to 0$, we see that by Lemma~\ref{lem::mean_loops_lcr} and the
  finiteness of $\E[|\xi_\CL|]$, the sum on the right-hand side goes to 0
in $L^1$ and hence in probability as $\eps \to 0$. Furthermore, we claim
that
\begin{equation*}
  \int_D \left[\acute{h}_{\eps|\varphi'(z)|}(\varphi(z)) -
  \acute{h}_{\eps}(\varphi(z))\right]g(z)\,dz \to 0
\end{equation*}
in probability as $\eps \to 0$. To see this, we write the difference in
square brackets as
\[
\acute{h}_{\eps|\varphi'(z)|}(\varphi(z)) - \acute{h}_{C\eps}(\varphi(z)) +
\acute{h}_{C\eps}(\varphi(z)) - \acute{h}_{\eps}(\varphi(z)),
\]
where $C$ is an upper bound for $|\varphi'(z)|$ as $z$ ranges over the
support of $g$.  Note that $\int_D
\left[\acute{h}_{C\eps}(\varphi(z)) -
  \acute{h}_{\eps}(\varphi(z))\right]g(z)\,dz \to 0$ in probability because
for all $0<\eps'<\eps$ and $\psi \in C_c^\infty(D)$, we have
\begin{align*}
\E&\|\psi h_\eps - \psi h_{\eps'}\|^2_{H^{-d-\delta}(\T^d)}
= \sum_{k\in \Z^d} \E|\widehat{\psi h_\eps - \psi h_{\eps'}}(k)|^2
\ang{k}^{-2(d+\delta)} \\
&\leq \sum_{k\in \Z^d} \|\psi \|^2_{L^\infty(\R^d)}
\iint_{D_\psi^2} \big|\E[ (h_\eps(x) - h_{\eps'}(x))(h_\eps(y) -
h_{\eps'}(y))]\big| \, dx \,dy \ang{k}^{-2(d+\delta)} \\
&\leq \eps^{\Omega(1)}/\delta;
\end{align*}
see \eqref{eqn::diff_bound} for more details. The same calculation along
with Theorem~\ref{thm::var-field-diff} show that
\enlargethispage{6pt}
\[
\int_D
\left[\acute{h}_{C\eps}(\varphi(z)) -
  \acute{h}_{\eps|\varphi'(z)|}(\varphi(z))\right]g(z)\,dz \to 0,
\]
in probability. It follows that $\ang{h,g} = \ang{\acute{h}\circ\varphi,
  g}$ for all $g\in C_c^\infty(D)$, as desired.
\end{proof}

\section{Step nesting} \label{sec::step}
\makeatletter{}In this section we prove Theorem~\ref{thm::existence_weighted_alt}.
Suppose that $D$ is a proper simply connected domain, and let $\Gamma$ be a
$\CLE_\kappa$ in $D$. Let $\mu$ be a probability measure with finite second
moment and zero mean, and define
\[
\mathfrak{h}_n(z) = \sum_{k=1}^n \xi_{\Loop_k(z)}, \quad n\in \N\,.
\]
We call $(\mathfrak{h}_n)_{n\in \N}$ the \textit{step nesting sequence\/}
associated with $\Gamma$ and $(\xi_\Loop)_{\Loop\in\Gamma}$.

\begin{lemma}\label{lem::Nzw}
  For each $\kappa\in(8/3,8)$ there are positive constants $c_1$, $c_2$, and $c_3$
  (depending on $\kappa$) such that for any simply connected proper
  domain $D\subsetneq\C$ and points $z,w\in D$, for a $\CLE_\kappa$ in $D$,
  \[\Pr\left[\Loopcount_{z,w}\geq c_1 \log \frac{\confrad(z;D)}{|z-w|}+c_2 j+c_3\right]\leq\exp[-j]\,.\]
\end{lemma}
\begin{proof}
  Let $X_i$ be i.i.d.\ copies of the log conformal radius distribution,
  and let $T_\ell=\sum_{i=1}^\ell X_i$.  Then
\begin{align*}
\Pr[T_\ell\leq t] &\leq \E[e^{-X}]^\ell e^{t} \\
\Pr[T_\ell\leq \log(\confrad(z;D)/|z-w|)]
 &\leq \E[e^{-X}]^\ell \frac{\confrad(z;D)}{|z-w|}\,.
\end{align*}
If $T_\ell > \log(\confrad(z;D)/|z-w|)$, then $J_{z,|z-w|}^\cap\leq\ell$.
But $\Loopcount_{z,w}< J_{z,|z-w|}^\subset$,
and by Corollary~\ref{cor::loop_contain_stoch_dom},
$J_{z,|z-w|}^\subset-J_{z,|z-w|}^\cap$ has exponential tails.
\end{proof}

\begin{proof}[Proof of Theorem~\ref{thm::existence_weighted_alt}]
  We check that \eqref{eq::cov_decay} holds with $f_n = \mathfrak{h}_n$.
  Writing out each difference as a sum of loop weights and using the
  linearity of expectation, we calculate for $0\leq m\leq n$ and $z,w\in
  D$,
  \begin{align*}
    \E[(\mathfrak{h}_m(z)-\mathfrak{h}_n(z))(\mathfrak{h}_m(w)-\mathfrak{h}_n(w))]
&= \sigma^2 \sum_{k=m+1}^n \P[\Loopcount_{z,w}\geq k]\,.
  \end{align*}

  Let $\delta(z)$ be the value for which
  $c_1\log(\confrad(z;D)/\delta(z))+c_3=k$, where $c_1$ and $c_3$ are as in
  Lemma~\ref{lem::Nzw}.  Let $K$ be compact, and let $\delta=\max_{z\in
    K}\delta(z)$.  Then
  \begin{equation} \label{eq::delta-k-K-D}
    \delta \leq \exp[-\Theta(k)] \times \sup_{z\in K} \dist(z,\partial D)
  \end{equation}
  and
  \begin{equation} \label{eq::KtimesK}
    \iint\limits_{\substack{K\times K\\|z-w|\geq\delta}} \Pr[\Loopcount_{z,w}\geq k] \,dz\,dw
    \leq \exp(-k) \times\area(K)^2\,.
  \end{equation}
  The integral of $\P[\Loopcount_{z,w}\geq k]$ over $z,w$ which are
  closer than $\delta$ is controlled by virtue of the small volume of the domain of
  integration:
  \begin{equation} \label{eq::Delta_n}
    \iint\limits_{\substack{K\times K\\|z-w|\leq \delta}}\P[\Loopcount_{z,w}\geq k]
    \,dz\,dw \leq \delta^2\times\area(K)\,.
  \end{equation}
  Putting together \eqref{eq::delta-k-K-D}, \eqref{eq::KtimesK} and \eqref{eq::Delta_n} establishes
  \begin{align} \label{eq::twoterms}
    \iint\limits_{K\times K}\P\left[\Loopcount_{z,w}\geq k\right]\,dz\,dw \leq
    \exp[-\Theta(k)]\times C_{K,D}
  \end{align}
  as $k\to\infty$.

  Having proved \eqref{eq::twoterms}, we may appeal to
  Proposition~\ref{prop::convergence_Hminusd} and conclude that
  $\mathfrak{h}_n$ converges almost surely to a limiting random variable
  $\mathfrak{h}$ taking values in $\Hloc^{-2-\delta}(D)$.

  Since each $\mathfrak{h}_n$ is determined by $\Gamma$ and
  $(\xi_\Loop)_{\Loop\in \Gamma}$, the same is true of
  $\mathfrak{h}$. Similarly, for each $n\in \N$, $\mathfrak{h}_n$ inherits
  conformal invariance from the underlying $\CLE_\kappa$. It follows that
  $\mathfrak{h}$ is conformally invariant as well.
\end{proof}

The following proposition shows that if the weight distribution $\mu$ has
zero mean, then the step nesting field $\mathfrak{h}$ and the usual nesting
field $h$ are equal.

\begin{proposition} \label{prop::fieldsequal} Suppose that $D \subsetneq
  \C$ is a simply connected domain, and let $\mu$ be a probability measure
  with finite second moment and zero mean. Let $\Gamma$ be a $\CLE_\kappa$
  in $D$, and let $(\xi_\Loop)_{\Loop\in \Gamma}$ be an
  i.i.d.\ sequence of $\mu$-distributed random variables. The weighted
  nesting field $h=h(\Gamma,(\xi_\Loop)_{\Loop\in \Gamma})$ from
  Theorem~\ref{thm::existence_weighted} and the step nesting field
  $\mathfrak{h}=\mathfrak{h}(\Gamma,(\xi_\Loop)_{\Loop\in \Gamma})$ from Theorems
  \ref{thm::existence_weighted} and \ref{thm::existence_weighted_alt} are
  almost surely equal.
\end{proposition}

\begin{proof}
  Let $g\in C_c^\infty(D)$, $\eps>0$ and $n\in \N$. By Fubini's theorem, we
  have
  \begin{align}\label{eq::twofields}
    \E[(&\ang{h_\eps,g}-\ang{\mathfrak{h}_n,g})^2] \\ \nonumber &= \iint_{D\times D}
    \E[(h_\eps(z) - \mathfrak{h}_n(z))(h_\eps(w) -
    \mathfrak{h}_n(w))]\,g(z)g(w)\,dz\,dw\,.
  \end{align}
  Applying the same technique as in \eqref{eq::h_diff}, we find that the
  expectation on the right-hand side of \eqref{eq::twofields} is bounded by
  $\sigma^2$ times the expectation of the number $N_{z,w}(n,\eps)$ of loops
  $\Loop$ satisfying both of the following conditions:
  \begin{enumerate}
  \item $\Loop$ surrounds $B_z(\eps)$ or $\Loop$ is among the $n$ outermost
    loops surrounding~$z$, but not both.
  \item $\Loop$ surrounds $B_w(\eps)$ or $\Loop$ is among the $n$ outermost
    loops surrounding~$w$, but not both.
  \end{enumerate}
  Using Fatou's lemma and \eqref{eq::twofields}, we find that
\begin{align*}
  \E[(\ang{h,g}-\ang{\mathfrak{h},g})^2] &= \E\left[\lim_{\eps\to
      0}\lim_{n\to\infty} (\ang{h_\eps,g}-\ang{\mathfrak{h_n},g})^2\right] \\
  &\leq \liminf_{\eps\to 0}\liminf_{n\to\infty} \E[ (\ang{h_\eps,g}-\ang{\mathfrak{h_n},g})^2] \\
  &\leq \liminf_{\eps\to 0}\liminf_{n\to\infty} \iint_{D\times D}\E[N_{z,w}(n,\eps)]\,g(z)g(w) \,dz\,dw \\
  &\leq \limsup_{\eps\to 0}\limsup_{n\to\infty} \iint_{D\times D}\E[N_{z,w}(n,\eps)]\,g(z)g(w) \,dz\,dw\,. 
\end{align*}
If $z\neq w$, then $\Loopcount_{z,w}<\infty$ almost surely,
so $\E[N_{z,w}(n,\eps)]$ tends to 0 as $\eps \to 0$ and
$n\to\infty$. Furthermore, the observation $N_{z,w}(n,\eps) \leq
\Loopcount_{z,w}$ implies by Theorem~\ref{thm::surround_two_point_bound}
that $\E[N_{z,w}(n,\eps)]$ is bounded by $\lptyp\log|z-w|^{-1}+\const$
independently of $n$ and $\eps$.  Since
$(z,w)\mapsto\E[N_{z,w}(n,\eps)]g(z)g(w)$ is dominated by the integrable
function $(\lptyp \log|z-w|^{-1}+\const)g(w)g(w)$, we may apply the reverse
Fatou lemma to obtain
\begin{align*}
  \E[(\ang{h,g}-\ang{\mathfrak{h},g})^2] &\leq \iint_{D\times D}
  \limsup_{\eps\to
    0}\limsup_{n\to\infty} \E[N_{z,w}(n,\eps)]\,g(z)g(w) \,dz\,dw \\  &= 0\,,
\end{align*}
which implies   \begin{equation}
    \label{eq::equal_as}
    \ang{h,g}=\ang{\mathfrak{h},g}
  \end{equation}
  almost surely.  The space $C_c^\infty(\C)$ is separable
  \cite[Exercise~1.13.6]{tao-epsilon}, which implies that $\C_c^\infty(D)$
  is also separable.  To see this, consider a given countable dense subset
  of $C_c^\infty(\C)$.  Any sufficiently small neighborhood of a point in
  $C_c^\infty(D)$ is open in $C_c^\infty(\C)$, and therefore intersects the
  countable dense set.  Therefore, we may apply \eqref{eq::equal_as} to a
  countable dense subset of $C_c^\infty(D)$ to conclude that $h = \mathfrak{h}$
  almost surely.
\end{proof}

\section{Further questions}
\label{sec::questions}

\makeatletter{}\newcounter{question}
	\setcounter{question}{1}

\newcommand{\ques}{\thequestion\refstepcounter{question}}

\smallskip
$\ $\newline

\noindent{\bf Question \ques.}  Suppose that $h$ is the nesting field associated with a $\CLE_\kappa$ process and weight distribution $\mu$.  For each $\eps > 0$ and $z \in D$, let $A_z(\eps)$ be the average of $h$ on the disk $B(z,\eps)$.  Is it true that the set of extremes of $A_z(\eps)$, i.e., points where either $A_z(\eps)$ has unusually slow or fast growth as $\eps \to 0$, is the same as that for $\SLoopcount_z(\eps)$?
\smallskip

\noindent{\bf Question \ques\hypertarget{ques::deterministic}.}  When $\kappa=4$ and $\mu$ is the Bernoulli distribution, the nesting field $h$ is a GFF on $D$.  In this case, it follows from \cite{MS_CLE} that the underlying $\CLE_4$ is a deterministic function of $h$.  Does a similar statement hold for $\kappa \in (8/3,4]$?  For $\kappa \in (4,8)$, we do not expect this to hold because we do not believe that it is possible to determine the outermost loops of such a $\CLE_\kappa$ given the union of the outermost loops as a random \textit{set}.  Nevertheless, is the union of all loops, viewed as a subset of $D$ and its prime ends, determined by the (weighted) nesting field?  \smallskip

\noindent{\bf Question \ques.}  When $\kappa=4$ and $\mu$ is the Bernoulli distribution, then the nesting field is a Gaussian process (in particular, a GFF).  We do not expect this to hold with the Bernoulli distribution for any $\kappa\neq 4$.  Do there exist other values of $\kappa \in (8/3,8)$ and weight distributions $\mu$ such that the corresponding nesting field is also Gaussian?
\smallskip

\noindent{\bf Question \ques.}  Does the nesting field in general satisfy a spatial Markov property which is similar to that of the GFF?  Is there a type of Markovian characterization for the nesting field which is analogous to that for $\CLE$ \cite{SW_CLE,SHE_CLE}?
The existence of a spatial Markov property for the nesting field is natural in view of the conjectured convergence of discrete models which possess a spatial Markov property to $\CLE_\kappa$.

\noindent{\bf Question \ques.} There are several discrete loop models which are known to converge to CLE.  Do their nesting fields converge to the nesting field of CLE?

\appendix

\section{Rapid convergence to full-plane CLE} \label{sec:converge-fullplane}

In this appendix we prove that $\CLE_\kappa$ (for $8/3<\kappa<8$) in
large domains rapidly converges to a full-plane version of
$\CLE_\kappa$.  This was proved in \cite{kemppainen-werner} when
$\kappa\leq 4$ using the loop-soup characterization of CLE.

For a collection $\Gamma$ of nested noncrossing loops in $\C$, let
$\Gamma|_{B(z,r)^+}$ denote the collection of loops in $\Gamma$ which
are in the connected component of
$\C\setminus\{\CL\in\Gamma:\text{$\CL$ surrounds $B(z,r)$}\}$
containing $z$.  If $\Gamma$ is a $\CLE_\kappa$ in a proper simply
connected domain containing $B(z,r)$, then $\Gamma|_{B(z,r)^+} =
\Gamma|_{U_z^{J^\cap_{z,r}-1}}$.

\begin{theorem}
  \label{thm::full-plane}
  For $\kappa\in(8/3,8)$ there is a unique measure on nested noncrossing
  loops in $\C$, ``full-plane $\CLE_\kappa$'', to which $\CLE_\kappa$'s on
  large domains $D$ rapidly converge in the following sense.  There are
  constants $C >0$ and $\alpha>0$ (depending on $\kappa$) such that for any
  $z\in\C$, $r>0$, and simply connected proper domain~$D$ containing
  $B(z,r)$, a full-plane $\CLE_\kappa$ $\Gamma_\C$ and a $\CLE_\kappa$
  $\Gamma_D$ on $D$ can be coupled so that with probability at least $1-C
  (r/\dist(z,\partial D))^{\alpha}$, there is a conformal map $\varphi$
  from $\Gamma_\C|_{B(z,r)^+}$ to $\Gamma_D|_{B(z,r)^+}$ which has low
  distortion in the sense that $|\varphi'(z)-1| < C (r/\dist(z,\partial
  D))^{\alpha}$ on $\Gamma_\C|_{B(z,r)^+}$.  Full-plane $\CLE_\kappa$ is
  invariant under scalings, translations, and rotations.
\end{theorem}

For $\kappa\leq 4$ Kemppainen and Werner showed that full-plane $\CLE_\kappa$
is also invariant under the map $z\mapsto 1/z$ \cite{kemppainen-werner}.

\begin{proof}
  We first prove for that $x>0$, the stated estimates hold for $z=0$ and
  $r=1$, with $\C$ and $D$ replaced by any two proper simply connected
  domains $D_1$ and $D_2$ which both contain the ball $B(0,e^x)$.

  For $i\in\{1,2\}$, let $\Gamma_i$ denote the $\CLE_\kappa$ on $D_i$.  Let
  $\pp^{(i)}_j=-\log\confrad(0,\CL_0^j(\Gamma_i))$.  Note that
  $\pp^{(i)}_0\leq-x$ for $i\in\{1,2\}$. Furthermore,
  $\{\pp^{(i)}_{j+1}-\pp^{(i)}_j\}_{j\in\N}$ is an i.i.d.\ positive
  sequence whose terms have a continuous distribution with exponential
  tails \cite{SSW}. Therefore, by Lemma~\ref{lem::coalesce}, there is
  a stationary point process $\pp^{(0)}$ on $\R$ with i.i.d.\ increments
  from this same distribution, and the sequences $\pp^{(1)}$ and
  $\pp^{(2)}$ can be coupled to $\pp^{(0)}$ so that
  $\pp^{(i)}\cap(-\frac34 x,\infty) = \pp^{(0)}\cap(-\frac34 x,\infty)$ for
  $i\in\{1,2\}$, except with probability exponentially small in $x$.

  Let
  \begin{equation} \label{eq::defa}
    a\colonequals \min\left(\pp^{(0)}\cap\left(-\frac34x,\infty\right)\right).
  \end{equation}
By Lemma~\ref{lem::firsthitting},
$a\in(-\frac34 x,-\frac12 x)$ except with probability exponentially small in $x$.

We shall couple the two $\CLE_\kappa$ processes $\Gamma_1$ and $\Gamma_2$
as follows.  First we generate the random point process $\pp^{(0)}$.  Then
we sample the negative log conformal radii of the loops of $\Gamma_1$ and
$\Gamma_2$ surrounding $0$, so as to maximize the probability that these
coincide with $\pp^{(0)}$ on $(-\frac34x,\infty)$.  If either $\pp^{(1)}$
or $\pp^{(2)}$ does not coincide with $\pp^{(0)}$ on $(-\frac34x,\infty)$,
then we may complete the construction of $\Gamma_1$ and $\Gamma_2$
independently. Otherwise, we construct $\Gamma_1$ and $\Gamma_2$ up to and
including the loop with conformal radius $e^{-a}$, where $a$ is defined in
\eqref{eq::defa}. Let $\CL_a^{(i)}$ denote the loop of $\Gamma_i$ whose
negative log conformal radius (seen from $0$) is $a$, and let $U_a^{(i)}$
denote the connected component of $\C\setminus\CL_a^{(i)}$ containing $0$.
Then we sample a random $\CLE_\kappa$ $\Gamma_\D$ of the unit disk $\D$
which is independent of $a$ and the portions of $\Gamma_1$ and $\Gamma_2$
constructed thus far.  (We can either take $\Gamma_\D$ to be independent of
$\pp^{(0)}$, or so that the negative log conformal radii of $\Gamma_\D$'s
loops surrounding $0$ coincide with $(\pp^{(0)}-a)\cap(0,\infty)$.)  Then
we let $\psi^{(i)}$ be the conformal map from $\D$ to $U_a^{(i)}$ with
  $\psi^{(i)}(0)=0$ and $(\psi^{(i)})'(0)>0$, and set the restriction of
$\Gamma_i$ to $U_a^{(i)}$ to be $\psi^{(i)}(\Gamma_\D)$.  If there are any
bounded connected components of $\C\setminus\CL_a^{(i)}$ other than
$U_a^{(i)}$, then we generate the restriction of $\Gamma_i$ to these
components independently of everything else generated thus far.  The
resulting loop processes $\Gamma_i$ are distributed according to the
conformal loop ensemble on $D_i$, and have been coupled to be similar near
$0$.

  Let $\psi = \psi^{(2)} \circ (\psi^{(1)})^{-1}$ be the conformal map
  from $U_a^{(1)}$ to $U_a^{(2)}$ for which $\psi(0)=0$ and
  $\psi'(0)=1$.  Assuming
  $a \in (-\frac34 x,-\frac12 x)$ and $a\in\lambda^{(1)}$ and $a\in\lambda^{(2)}$,
   the Koebe distortion theorem implies that on $B(0,e^{x/4})$, $|\psi'-1|$ is
  exponentially small in $x$.

By \cite[Lemma~\ref{I-lem::annulus-loop}]{extremes}, except with probability exponentially small in $x$,
both $\Gamma_1$ and $\Gamma_2$ contain a loop surrounding $B(0,1)$ which is
contained in $B(0,e^{x/4})$.
  It is possible that $\psi$ maps a loop of $\Gamma_1$ surrounding $\D$
  to a loop of $\Gamma_2$ intersecting $\D$ or vice versa.  But since
  $\psi$ has exponentially low distortion, any such loop would have to have
  inradius exponentially close to~$1$.  The expected number of loops of
  $\Gamma_1$ with negative log conformal radius between $-\log4$ and $1$ is
  bounded by a constant, so by the Koebe quarter theorem, the expected
  number of loops of $\Gamma_1$ with inradius between $1/e$ and $1$ is
  bounded by a constant.  Let $D_3 = e^u D_1$ be a third domain, where $u$
  is independent of everything else and uniformly distributed on $(0,1)$.
  It is evident that $e^u \Gamma_1$ has no loop with inradius exponentially
  (in $x$) close to~$1$, except with probability exponentially in $x$.  On
  the other hand, we can couple a $\CLE_\kappa$ on $D_3$ to $\Gamma_1$ in
  the same manner that we did for domain~$D_2$, and deduce that $\Gamma_1$
  must also have no loop with inradius exponentially close to~$1$, except
  with probability exponentially small in $x$.
We conclude that it is exponentially unlikely
  for there to be a loop of $\Gamma_1$ surrounding $B(0,1)$ which $\psi$
  maps to a loop of $\Gamma_2$ intersecting $B(0,1)$ or vice versa.  Thus
  $\psi(\Gamma_1|_{B(z,1)^+})=\Gamma_2|_{B(z,1)^+}$,
  except with probability exponentially small in $x$.

  The corresponding estimates for general $z$ and $r$ and domains
  $D_1$ and $D_2$ containing $B(z,r)$ follows from the conformal
  invariance of $\CLE_\kappa$.

  Given the above estimates for any two proper simply connected domains
  which contain a sufficiently large ball around the origin, it is not
  difficult to take a limit.  For some sufficiently large constant $k_0$
  (which depends on $\kappa$), we let $\Gamma_k$ be a $\CLE_\kappa$ in the
  domain~$B(0,e^k)$, where $k\geq k_0$ is an integer.  For each $k$, we
  couple $\Gamma_{k+1}$ and $\Gamma_k$ as described above.  With
  probability $1$ all but finitely many of the couplings have that
  $\Gamma_{k+1}|_{B(0,e^{k/2})^+}=\psi_k(\Gamma_k|_{B(0,e^{k/2})^+})$
  for a low-distortion conformal map $\psi_k$, so suppose that this is
  the case for all $k\geq\ell$.  The conformal maps $\psi_k$ (for
  $k\geq\ell$) approach the identity map sufficiently rapidly that for each
  $m\geq\ell$, the infinite composition $\cdots\circ\psi_{m+1}\circ\psi_m$
  is well defined and converges uniformly on compact subsets to a limiting
  conformal map with distortion exponentially small in $m$.  We define
  $\Gamma_\C|_{B(0,e^{m/2})^+}$ to be the image of $\Gamma_m|_{B(0,e^{m/2})^+}$
  under this infinite composition.  These satisfy the consistency condition
  $\Gamma_\C|_{B(0,\exp(m_1/2))^+} \subseteq \Gamma_\C|_{B(0,\exp(m_2/2))^+}$ for
  $m_2\geq m_1\geq\ell$, so then we define $\Gamma_\C = \bigcup_{m\geq\ell}
  \Gamma_\C|_{B(0,e^{m/2})^+}$.  For any other proper simply connected domain
  $D$ containing a sufficiently large ball around the origin, we couple
  $\Gamma_D$ to $\Gamma_{\lfloor \log\dist(0,\partial D)\rfloor}$ as
  described above, and with high probability it will be close to
  $\Gamma_\C$ in the sense described in the theorem.

  It is evident from this construction of $\Gamma_\C$ that it is
  rotationally invariant around $0$.  Next we check that $\Gamma_\C$ is
  invariant under transformations of the form $z\mapsto A z +C$ where
  $A,C\in\C$ and $A\neq 0$.  Note that $A \Gamma_\C+C$ restricted to a ball
  $B(0,r)$ is arbitrarily well approximated by CLE on $B(C,A 2^k)$ for
  sufficiently large $k$.  But by the coupling for simply connected proper
  domains, the CLEs on $B(C,A 2^k)$ and $B(0,2^k)$ restricted to $B(0,r)$
  approximate each other arbitrarily well for sufficiently large $k$, and
  by construction, $\Gamma_\C$ restricted to $B(0,r)$ is arbitrarily well
  approximated by CLE on $B(0,2^k)$ restricted to $B(0,r)$ when $k$ is
  sufficiently large.  Thus full-plane CLE is invariant under affine
  transformations.

\enlargethispage{24pt}
  Finally, if there were more than one loop measure that approximates CLE
  on simply connected proper domains in the sense of the theorem, then for
  a sufficiently large ball the measures would be different within the
  ball. Since for some sufficiently large proper simply connected domain
  $D$, CLE on $D$ restricted to the ball would be well-approximated by both
  measures, we conclude that full-plane CLE is unique.
\end{proof}

\section*{Notation}
\label{sec::notation}

\makeatletter{}We use the following notation.

\begin{itemize}
\item $D$ is a simply connected proper domain in $\C$, i.e.\ $\varnothing\subsetneq D\subsetneq\C$ (p. \pageref{not::domD})
\item $\Gamma$ denotes a $\CLE_\kappa$ process on $D$
  (p. \pageref{not::CLE})
\item $\Loopcount_z(\eps)$ is the number of loops of $\Gamma$ which
  surround $B(z,\eps)$ (p.\ \pageref{not::Loopcount})
\item $\Loop_z^j$ is the $j$th loop of $\Gamma$ which surrounds $z$
  (p. \pageref{not::loop})
\item $U_z^j$ is the connected component of $D \setminus \Loop_z^j$ which
  contains $z$ (p. \pageref{not::loop})
\item $\SLoopcount_z(\eps)$ is the sum of the loop weights over the loops
  of $\Gamma$ which surround $B(z,\eps)$
  (\eqref{eqn::normalized_loop_count} on p.\
  \pageref{eqn::normalized_loop_count})
\item $\mu$ is the weight distribution on the loops (p.\ \pageref{not::mu})
\item $G_D(z,w)$ is the Green's function for the Dirichlet Laplacian
  $-\Delta$ on $D$ (p. \pageref{not::greens})
\item  $h_z(\eps) = \SLoopcount_z(\eps) - \E[ \SLoopcount_z(\eps) ]$ (p.\ \pageref{not::h_eps})
\item $\Loopcount_{z,w}(\eps)$ is the number of loops of $\Gamma$ which
  surround both $B(z,\eps)$ and $B(w,\eps)$ (p.\ \pageref{not::Loopcount_zw})
\item $J_{z,r}^\cap$ is the index of the first loop of $\Loop_z$ which
  intersects $B(z,r)$ (\eqref{eq::Jcapsubset} on p.\ \pageref{not::Jcapsubset}).
\item $J_{z,r}^\subset$ is the index of the first loop of $\Loop_z$ which
  is contained in $B(z,r)$ (\eqref{eq::Jcapsubset} on p.\
  \pageref{not::Jcapsubset}).
\item $\Gamma_z(\eps)$ is the set of loops of $\Gamma$ which surround
  $B(z,\eps)$ (p.\ \pageref{not::Gamma_z})
\item $\rc{X}=X-\E[X]$ for any integrable random variable $X$ (p.\
  \pageref{eqn::loopcount_bar_def})
\item $G_D^{\kappa,\eps}(z,w)$ is the expected number of CLE loops
  surrounding $z$ and $w$ but not neither $B(z,\eps)$ nor $B(w,\eps)$ (p.\
  \pageref{not::GDke})
\end{itemize}

\section*{Acknowledgements}

Both JM and SSW thank the hospitality of the Theory Group at Microsoft Research, where part of the research for this work was completed.  JM's work was partially supported by DMS-1204894 and SSW's work was partially supported by an NSF Graduate Research Fellowship, award No.~1122374.

\makeatletter
\def\@rst #1 #2other{#1}
\renewcommand\MR[1]{\relax\ifhmode\unskip\spacefactor3000 \space\fi
  \MRhref{\expandafter\@rst #1 other}{#1}}
\newcommand{\MRhref}[2]{\href{http://www.ams.org/mathscinet-getitem?mr=#1}{MR#1}}
\makeatother


\end{document}